\documentclass[reqno,11pt]{amsart}
\usepackage{amsthm,amsfonts,amssymb,euscript,mathrsfs,graphics,color,amsmath,amssymb,latexsym,marginnote, bigints}%,mathabx}
\usepackage{soul}
\usepackage{graphicx}
\usepackage{multicol,bbm}
\usepackage{stmaryrd}
\usepackage{amsthm,amsfonts,amssymb,euscript,mathrsfs,graphics,color,amsmath,amssymb,latexsym,marginnote,esint}
\usepackage[hidelinks]{hyperref}
\usepackage{comment}
\usepackage[margin=1.0in]{geometry}
\numberwithin{equation}{section}

\def\al{\alpha}

\def\ga{\gamma}

\def\De{\Delta}
\def\ep{\epsilon}

\def\om{\omega}
\def\Om{\Omega}
\def\th{\theta}

\def\ph{\varphi}

\def\FF{{\mathcal F}}

\def\HH{{\mathcal H}}

\def\R{\mathbb{R}}
\def\Z{\mathbb{Z}}
\def \T {\mathbb{T}}
\def \D {\mathbb{D}}

\def \I {I} %identity
\def \P {\mathbb{P}}

\def\pr{{\partial}}

\def\c{\cdot}

\def\bar{\overline}

\def\div{{\,\nabla \c\,}}
\def\curl{{\,\mbox{curl}\,}}

\def\tr{\mbox{tr}}

\newcommand{\parent}[1]{\left ( {#1} \right)}
\newcommand{\corch}[1]{\left[{#1}\right]}% Corchete
\newcommand{\llav}[1]{\left\lbrace{#1} \right \rbrace} % Llaves.

\newcommand{\abs}[1]{\left \vert {#1} \right \vert}
\newcommand{\norm}[1]{\|{#1}\|}
\newcommand{\normp}[2]{{\norm{#1}}_{L^#2}}
\newcommand{\produ}[2]{\left<{#1},{#2}\right>} % Producto escalar

\newcommand{\apli}[3]{{#1}:\apl{#2}{#3}} % Map with name.
\newcommand{\apl}[2]{{#1}\longrightarrow {#2}}

\newcommand{\mat}[1]{{\partial_t{#1}}+u \cdot \nabla{#1}}
\newcommand{\matt}[1]{u \cdot \nabla{#1}}

\definecolor{bluegreen}{rgb}{0.0, 0.3, 0.9}

\usepackage[normalem]{ulem}

\newtheorem{theorem}{Theorem}[section]
\newtheorem{lemma}[theorem]{Lemma}
\newtheorem{proposition}[theorem]{Proposition}

\newtheorem{definition}[theorem]{Definition}
\newtheorem{remark}[theorem]{Remark}

\newcommand{\Lp}[1]{L^{#1}(\R^2)}
\newcommand{\LpT}[1]{L^{#1}(\T^2)}
\def\Hgam{\dot{H}^{\gamma+s}(\R^2)}
\def\HgamT{H^{\gamma+s}(\T^2)}
\def\Cga{\dot{C}^\ga}

\def \srho{\sqrt{\rho}}
\def \sr {\sqrt{\rho_0}}
\def \smu {\sqrt{\mu}~}
\def \smuo {\sqrt{\mu_0}~}

\def \du{D_t u}
\def \dv{D_t v}
\def \dr{D_t \rho}
\def \dm {D_t \mu}

\title[On 2D N-S free boundary: nonnegative density and small viscosity contrast]{On 2D Navier-Stokes free boundary: \\nonnegative density and small viscosity contrast}
\author{Francisco Gancedo, Eduardo Garc\'ia-Ju\'arez, Paula Luna-Velasco}

\address{F. Gancedo: Departamento de An\'alisis Matem\'atico \& IMUS, Universidad de Sevilla, C/Tarfia s/n, Campus Reina Mercedes, 41012, Sevilla, Spain and
School of Mathematics, Institute for Advanced Study, 1 Einstein Dr., Princeton, NJ 08540, USA. \href{mailto:fgancedo@us.es}{fgancedo@us.es}}

\address{E. García-Juárez: Departamento de An\'alisis Matem\'atico \& IMUS, Universidad de Sevilla, C/Tarfia s/n, Campus Reina Mercedes, 41012, Sevilla, Spain. \href{mailto:egarcia12@us.es}{egarcia12@us.es}
}

\address{P. Luna-Velasco: Departamento de Análisis Matemático, Universidad de Granada, Avenida Fuen\-te\-nue\-va S/N, 18071, Granada, Spain and Universidad de Sevilla, C/Tarfia s/n, Campus Reina Mercedes, 41012, Sevilla, Spain.
\href{mailto:pluna@us.es}{pluna@us.es}}

\begin{document}

\begin{abstract}
This paper is concerned with the evolution of two incompressible, immiscible fluids in two dimensions governed by the inhomogeneous Navier-Stokes equations.  We prove global-in-time well-posedness, establishing the preservation of the natural $C^{1+\gamma}$ Hölder regularity of the free boundary, for $0<\gamma<1$. This is the first result that allows for nonnegative density driven by a low-regularity initial velocity, while also remaining valid in the presence of a small viscosity jump.
\end{abstract}

 \maketitle
 
 % Index
\setcounter{tocdepth}{1}
\tableofcontents

%--------------------------------------------------------------
%--------------------------------------------------------------

 %Sections
 
\section{Introduction}
The motion of two-dimensional, incompressible, viscous fluids with variable density is governed by the inhomogeneous Navier--Stokes equations:
\begin{equation}\label{VJ}
\left\{
    \begin{aligned}
         D_t\rho&=0, \\
         \rho D_tu &= \nabla \c \parent{\mu \D u - P\I},  \\
         \nabla \c u &=0,\\
         (\rho,u)|_{t=0}&=(\rho_0,u_0).
    \end{aligned}
    \right.
\end{equation}
Above, $u=u(t,x)\in\mathbb{R}^2$ is the velocity field, and the scalar functions $\rho=\rho(t,x)$, $\mu=\mu(t,x)$, $P=P(t,x)$ denote the fluid's density, viscosity, and pressure, respectively. The  total derivative is denoted by the operator $D_t=\partial_t + u\cdot\nabla $, the tensor $\D u$ denotes the symmetric part of the gradient, 
$\D_{j,k} u=\partial_ju_k+\partial_k u_j,\, j,k\in\{1,2\},$
and $\I$ is the identity matrix in $\R^2$. 
The viscosity is given by  $\mu=\tilde{\mu}(\rho)$, where $\tilde{\mu}$ is a smooth function, and  is therefore transported with the flow:
\begin{equation*}
    D_t\mu=0.
\end{equation*}

In this paper, the main interest lies in the free boundary problem that arises when considering two immiscible fluids with different densities and viscosities. Suppose that initially the fluids occupy a bounded domain $D_0\subset \R^2$ and its open complement $D_0^c=\R^2\setminus \overline{D_0}$. 
Then, this scenario can be modeled using \eqref{VJ}, considering an initial density function of the form
\begin{equation}\label{initialdensitypatch}
\rho_0(x)=\rho^{in}_0(x)\mathbbm{1}_{D_0}(x)+\rho^{out}_0(x) \mathbbm{1}_{D_0^c}(x),    
\end{equation}
with $\mathbbm{1}_A$ the indicator function of the set $A$, and accordingly for the initial viscosity
\begin{equation}\label{mupatch}
\mu_0(x)=\mu^{in}_0(x) \mathbbm{1}_{D_0}(x)+\mu^{out}_0(x)\mathbbm{1}_{D_0^c}(x).
\end{equation}

Since both fluids evolve with the velocity field via the particle trajectories 
\begin{equation}\label{traj}
  \left\{
  \begin{aligned}
\partial_tX(t,y)&=u(t,X(t,y)),\\
X|_{t=0}&=y,
    \end{aligned}\right.
\end{equation}
the density structure is preserved, $\rho(t,X(t,y))=\rho_0(y)$. 
The classical free boundary physical conditions without capillarity are given by 
\begin{equation*}
    \left.\begin{array}{r}
         \llbracket u \rrbracket = 0\\
       \llbracket \mu \D u- P\I  \rrbracket n = 0 
    \end{array}
\right\}\text{  on  }\, \partial D(t),\mbox{  with  }D(t)=X(t,D_0),\\
\end{equation*}
where $\llbracket \cdot \rrbracket$
denotes the difference between the limiting values on either side of the domain
$D(t)$ and $n$ is the unit normal to $\partial D(t)$. These conditions are recovered by interpreting \eqref{VJ} in a weak sense, together with the regularity of the solution (see \cite{Gancedo18}).
A central question is then the dynamics and regularity of the boundary $\partial D(t)$ in the physically relevant setting where the initial kinetic energy is finite
$$
\int \rho_0(x)\abs{u_0(x)}^2dx < \infty.
$$

\subsection{State of the art.}
The system \eqref{VJ} was initially studied under the assumption of constant viscosity and positive density. The first results concerning the existence of strong solutions for smooth data appeared in \cite{Ladyzhenskaya78}. In the case of nonnegative density, the result \cite{Simon90} established the global existence of weak solutions with finite energy satisfying the classical weak energy inequality,
\begin{equation}\label{energy_ineq}
        \norm{\srho u}^2 + 2\int_{0}^{t}\norm{\smu \D u}^2 d\tau \leq \norm{\sqrt{\rho_0}u_0}^2.
\end{equation}
This result was later extended in \cite{Lions} to the case of variable viscosity. In this book, the so-called density patch problem was introduced, where the initial density is given by
$$
\rho_0(x)=\mathbbm{1}_{D_0}(x).
$$
The question posed was whether the regularity of \( D(t) \) is preserved over time. Theorem 2.1 in \cite{Lions} establishes that the density remains a volume-preserving patch; however, the previous question was left open.

\medskip
Global-in-time regularity for Navier-Stokes free boundary problems has been extensively studied, particularly considering the continuity of the stress tensor at the free boundary (see e.g. \cite{DenisovaSolonnikov21}). Local existence results were first obtained for one fluid in \cite{Solonnikov77, Beale81}, while global-in-time existence was proved for an almost horizontal viscous fluid subject to gravity lying above a bottom in \cite{Sylvester90, Tani95}. The decay rate of the solution in such scenarios has been analyzed in \cite{Guo2013, GuoTice2013, GuoTice13}. In the two-fluid case, global well-posedness and decay near the flat equilibrium have been established in \cite{Wang14}. These global-in-time well-posedness  results rely on 
a coordinate transformation to fix in time the free boundary, which requires smallness conditions on the initial data.

\medskip
A different perspective is to consider fluids using the inhomogeneous Navier-Stokes system. In the case with \textit{constant viscosity}, i.e. $\mu=1$ in \eqref{VJ}, global well-posedness was proved for continuous densities bounded away from zero in \cite{Danchin04, Danchin06}. The works \cite{Danchin12, Danchin13} allowed the density to be discontinuous across $C^1$ interfaces, although  assumptions of sufficiently small jumps and small initial velocity were still needed. The initial regularity required for the velocity was then lowered in \cite{Huang13}. Maintaining a lower and upper bound for the density, global well-posedness for $u_0\in H^s(\mathbb{R}^2)$, $s\in(0,1)$, was proved in \cite{Paicu13} without smallness conditions. In \cite{Liao16}, the persistence of \( W^{k,p}(\mathbb{R}^2) \) regularity of the boundary, for \( k \geq 3 \) and \( p \in (2,4) \), is established for initial patches of the form \eqref{initialdensitypatch}, where \( \rho_0^{in} \) and \( \rho_0^{out} \) are distinct but close positive constants.
 The propagation of  regularity for small density jumps and small initial velocity with critical regularity was shown in \cite{Danchin17}. Later, in \cite{Liao19}, the smallness condition in \cite{Liao16} was removed. In \cite{Gancedo18}, the propagation of \( C^{1+\gamma} \) regularity of the interface was proved for initial velocity with sharp regularity, \( u_0 \in H^{\gamma+s}(\mathbb{R}^2) \), where \( \gamma \in (0,1) \), \( s \in (0,1-\gamma) \), and without any smallness assumptions.

\medskip
Allowing the density to be merely nonnegative introduces significant challenges in the analysis of the system \eqref{VJ}. In particular, a main difficulty is the loss of parabolicity where $\rho(t,x) = 0$. In such regions, the fluid evolves by the Stokes's law, exhibiting an elliptic rather than parabolic structure. This degeneracy prevents the use of standard energy methods. The case of nonnegative density was first studied \cite{Cho04} by requiring
 the initial density to be sufficiently smooth, which precludes patch-type data. Moreover, the initial velocity $u_0$ needed to satisfy a compatibility condition
\begin{equation*}
    -\Delta u_0 + \nabla P_0 = \sr g,
\end{equation*}
for some $g \in L^2(\Omega)$ and $P_0 \in H^1(\Omega)$ in a bounded domain $\Omega\subset\mathbb{R}^2$. 
It was only very recently when the first result for just bounded nonnegative density appeared \cite{Danchin19}. The result assumes $u_0 \in H^1(\Omega)$, $\Omega = \mathbb{T}^2$ or a smooth bounded domain, without any compatibility condition. Taking advantage of the gain of regularity for the velocity,  uniqueness is proved using a Lagrangian approach (see also \cite{Danchin12}). Since the proof used certain Poincaré-type inequalities, it was not amenable for the whole space $\mathbb{R}^2$. 
Recovering the compatibility condition for the initial velocity,  the authors in \cite{Prange24} were the first to be able to deal with the density patch problem in the whole space $\mathbb{R}^2$ with nonnegative density. Then, the work \cite{Hao24a} removed the compatibility condition, solving the problem with $u_0\in \dot{H}^1(\mathbb{R}^2)$ and nonnegative, bounded density. The uniqueness part required additional conditions on the density, satisfied in particular by density patches. As in \cite{Prange24}, the proof of uniqueness follows a duality argument, inspired by works in the compressible case (see \cite{Hoff06, DanchinMucha23}). Interestingly, the authors also proved in a separate paper \cite{Hao24b} the uniqueness of weak solutions with critical initial velocity $u_0\in L^2(\mathbb{R}^2)$, only when the density is bounded from below.

\medskip
In the first part of this paper, we will solve the inhomogeneous Navier-Stokes for nonnegative density and initial velocity $u_0$ with low regularity (see Theorem \ref{Th1}). Therefore, our results address the degeneracy inherent in the problem, proving global existence and uniqueness under minimal assumptions on the initial density and velocity regularity, while recovering the appropriate parabolic regularity. In particular, $u_0\in\dot{H}^{\gamma+s}(\mathbb{R}^2)$, where $\ga\in(0,1)$, $s\in(0,1-\ga)$, which is sharp in the Sobolev scale for the propagation of $C^{1+\gamma}$ regularity of $\partial D(t)$. This allows us to solve the free boundary Navier–Stokes problem for fluids of different densities, not necessarily bounded from below. The regularity we assume on the initial velocity is the lowest known to ensure uniqueness without requiring a strictly positive density.

\medskip
We next consider the case of \textit{density-dependent viscosity}. Under the assumptions of merely bounded density, $u_0 \in H^1 {(\T^2)}$, and sufficiently small viscosity variation in $L^\infty(\T^2)$,  the solutions  where shown to satisfy that $u \in L^\infty(0,T; H^1(\T^2))$, $\srho \partial_t u \in L^2(0,T; L^2(\T^2))$, and $\rho, \mu \in L^\infty(0,T;L^\infty(\T^2))$ for all $T>0$ \cite{Desjardins97}. However, this regularity is not sufficient to propagate the regularity of interfaces.
Assuming \textit{positive} and bounded density, the first result showing propagation of regularity for patches with variable viscosity was given in \cite{Paicu20}.
Considering $u_0\in H^1(\mathbb{R}^2)$, the authors were able to propagate $H^{5/2}$ regularity of $\partial D(t)$, imposing additional striated regularity for the initial viscosity and smallness in the jump of viscosities.
 In \cite{Gancedo23}, the initial regularity was lowered to $u_0 \in \Hgam$ to propagate $C^{1+\gamma}$ regularity of $\partial D(t)$.
 Without small viscosity jump, global-in-time well-posedness can be obtained for small initial velocity \cite{HYZ2020}. More recently, global-in-time well-posedness has been shown for initial data satisfying a smallness condition that involves the viscosity jump together with the initial velocity \cite{Liao24}.

\medskip In the second part of the paper, we obtain the first result on the propagation of regularity for the free boundary Navier--Stokes equations for fluids with different densities and viscosities, without assuming a positive lower bound on the density and without any smallness condition on the velocity (see Theorem~\ref{th2}). The regularity of the initial data is optimal in the Sobolev scale for propagating the natural \(C^{1+\gamma}\) regularity and for obtaining the optimal parabolic regularity. The techniques developed in \cite{GANCEDO22} were not sufficient to address the nonnegative density case.

\subsection{Main Results} In this paper, we prove global-in-time well-posedness for the two dimensional density patch problem \eqref{VJ} in $\R^2$ with nonnegative density and constant viscosity. We also prove global well-posedness of \eqref{VJ} with nonnegative density and small viscosity jump in $\T^2$. We study the evolution of the two fluids as well as the evolution of the interface. 

The initial density $\rho_0$ is assumed to satisfy the uniform bound
$$
0\leq \rho_0(x)\leq \rho^M,
$$
allowing in particular for the presence of regions where $\rho_0(x) = 0$, where the fluid is governed by the Stokes' law.

\medskip
Specifically in the $\R^2$ setting, we further assume that $\rho_0$ satisfies one of the following conditions, as introduced in \cite{Hao24a}:
\begin{equation*}\tag{A}\label{H1}
    (1+\abs{x}^2)\rho_0\in\Lp{1},
\end{equation*}
\begin{equation*}\tag{B}\label{H2}
    \text{ There exists }R_0, c_0\in (0, \infty) \text{ such that } \int_{B(x_0,R_0)}\rho_0(x)dx\geq c_0 >0\, \textup{ for all } x_0\in\R^2.
\end{equation*}
We then establish the following result:
\begin{theorem}\label{Th1}
      Let $\mu=1$, $\parent{\rho_0,u_0}$ such that $0\leq \rho_0 \leq \rho^M$, $\rho_0 \not\equiv 0$, $\rho^M>0$ and satisfies \eqref{H1} or \eqref{H2}. Moreover, let $\sr u_0\in \Lp{2}$, $u_0\in \Hgam$ a divergence-free vector with $\ga\in(0,1)$, $s\in(0,1-\ga)$. 
    Then, the system \eqref{VJ} has a unique global solution $(\rho,u)$ satisfying the weak energy inequality \eqref{energy_ineq}, $0\leq \rho(t,x) \leq \rho^M$, and the following properties
     \begin{center}
     \begin{multicols}{2}
        $t^{\frac{1-(\ga+s)}{2}}\sqrt{\rho}D_t u\in L^2(0,T;\Lp{2})$,\\ 
        $t^{\frac{2-(\ga+s)}{2}}\sqrt{\rho}\du\in L^\infty(0,T;\Lp{2})$,\\
        $\nabla u \in L^1(0,T;\Lp{\infty})$,\\
        $t^{\frac{1-(\ga+s)}{2}}\nabla{u}\in L^\infty(0,T;\Lp{2})$, \\ 
        $t^{\frac{2-(\ga+s)}{2}}\nabla D_t u\in L^2(0,T;\Lp{2})$,\\
        $\nabla u \in L^1(0,T;\Cga(\R^2))$.
     \end{multicols}
     \end{center}
 \end{theorem}

 \begin{remark}\label{remark1.2}
Theorem~\ref{Th1} above allows us to consider the previously described free boundary setting. Indeed, let $D_0 \subset \mathbb{R}^2$ be a bounded domain of class $C^{1+\gamma}$, $\gamma \in (0,1)$, and let $0 \leq \rho_0 \not\equiv 0$ be the initial density given by \eqref{initialdensitypatch} with $\rho_0^{in}\in C^{\gamma}(\bar{D_0})$, $\rho_0^{out}\in C^{\gamma}(\mathbb{R}^2\setminus D_0)$. Then \eqref{VJ} has a unique solution $(\rho,u)$ that satisfies the same estimates as in Theorem \ref{Th1}. Moreover, the regularity of the velocity provides
    $$    \rho(t,x)=\rho^{in}(t,x)\mathbbm{1}_{D(t)}(x)+\rho^{out}(t,x)\mathbbm{1}_{D(t)^{c}}(x),\quad\mbox{where}\quad\rho(t,X(t,y))=\rho_0(y),$$
    $D(t)=X(t,D_0)$ with $X$ the particle trajectories \eqref{traj}  and $\partial D(t) \in C([0,T],C^{1+\ga})$.
 \end{remark}

\begin{remark} 
  Theorem~\ref{Th1} also applies in the periodic setting $\mathbb{T}^2$, in which case conditions~\eqref{H1} and~\eqref{H2} are automatically satisfied for bounded densities.
\end{remark}

The proof of this theorem is detailed in Section \ref{sect:INS constant vis.}.  
The first part is devoted to deriving a priori estimates. The lack of a lower bound for the density prevents the $L^2$ control of the velocity via standard energy estimates. In \cite{Danchin19, Prange24, Hao24a}, the  use of Poincaré-type inequalities (see Lemma \ref{Lemma3.7} or Lemma \ref{LemDanchin2}) is crucial to overcome this difficulty and to obtain the boundedness of $L^p$ norms. These inequalities require the control of the gradient, which  is bounded uniformly in time for initial data in $\dot{H}^1$. 
However, for initial data in $\dot{H}^{\gamma+s}$, applying these inequalities results in a loss of $1-\gamma-s$  derivatives. 
Due to the lower regularity of the initial velocity, we introduce time weights and apply an interpolation argument to propagate the regularity. A main difficulty arises from the fact that the interpolation is carried out between a Banach and a non-Banach space, a situation that is rarely addressed in the literature. Interpolations between different type of measures are needed to obtain the appropriate estimates. 
Another key step in the proof consists in establishing an $L^1$-in-time, Lipschitz-in-space estimate for the velocity field. To obtain such a bound, we rely on the explicit structure of the initial density, assumptions \eqref{H1} or \eqref{H2}, which allow us to control the spatial $L^m$ norm of the Laplacian of the velocity. In the case  \eqref{H1}, the available control of the velocity is only local in space, which forces us to adopt a localized regularization strategy locally in time and space. Finally, uniqueness is established through a duality argument.

\medskip We now consider the case in which the viscosity depends smoothly on the density; that is, $\mu_0 = \tilde{\mu}(\rho_0)$ for some smooth function $\tilde{\mu}$. In particular, the interest lies in discontinuous initial viscosities such as in \eqref{mupatch}.

\begin{theorem}\label{th2}
    Let $D_0\subset \T^2$ be a domain whose boundary $\partial D_0\in C^{1+\ga}$, $0<\ga<1$, is non-self-intersecting. Let $\rho_0^{in}\in C^{\ga}(\bar{D_0})$, $\rho_0^{out}\in C^\ga(\T^2 \setminus D_0)$, $0\leq \rho_0 \leq \rho^M$,
    and  $\mu_0=\tilde{\mu}(\rho_0)$ with $\tilde{\mu}$ smooth, $0< \mu^m\leq \mu_0\leq \mu^M$. Let the initial density be given by \eqref{initialdensitypatch} and $\sr u_0 \in \LpT{2}$, $u_0\in \HgamT$, $s\in(0,1-\ga)$ be a divergence-free vector field. Then, there exists $\delta>0$ such that if 

    \begin{equation}
    \Big\| 1 - \frac{\mu_0}{\bar{\mu}} \Big\|_{L^\infty} < \delta, \quad \text{with } \bar{\mu} = \frac{\mu^m + \mu^M}{2},
    \end{equation}
    the system \eqref{VJ} admits a unique global solution $(\rho,u)$ with $\mu = \tilde{\mu}(\rho)$ satisfying the weak energy inequality \eqref{energy_ineq} and
 \begin{center}
     \begin{multicols}{2}
        $t^{\frac{1-(\ga+s)}{2}}\sqrt{\rho}D_t u\in L^2(0,T;\LpT{2})$,\\ 
        $t^{\frac{2-(\ga+s)}{2}}\sqrt{\rho}\du\in L^\infty(0,T;\LpT{2})$,\\
        $\nabla u \in L^1(0,T;\LpT{\infty})$,\\
        $t^{\frac{1-(\ga+s)}{2}}\nabla{u}\in L^\infty(0,T;\LpT{2})$, \\ 
        $t^{\frac{2-(\ga+s)}{2}}\nabla D_t u\in L^2(0,T;\LpT{2})$,\\
        $\nabla u \in L^1(0,T;\Cga_D)$,
     \end{multicols}
     \end{center}

where $\Cga_{D}=\Cga(\bar{D(t)})\cup \Cga (\T^2\setminus D(t))$. Moreover,
$$\rho(t,x)=\rho^{in}(t,x) \mathbbm{1}_{D(t)}(x)+\rho^{out}(t,x) \mathbbm{1}_{D(t)^{c}}(x),\quad\mbox{where}
\quad    \rho(t,X(t,y))=\rho_0(y),$$  $D(t)=X(t,D_0)$ with $X$ the particle trajectories \eqref{traj} and $\partial D(t) \in C([0,T],C^{1+\ga})$.
\end{theorem}
The proof of this theorem is detailed in Section \ref{sect:viscosity contrast}. In the periodic setting, the availability of a Poincaré inequality and a Gagliardo–Nirenberg-type inequality (see Lemma \ref{LemDanchin2}) is a key ingredient in the analysis. While the strategy retains some elements from the constant-viscosity case, such as the use of time weights and interpolation, the presence of a non-constant, density-dependent viscosity introduces new challenges.
A central difficulty lies in the explicit treatment of the pressure and the fact that the velocity gradient is given implicitly by higher-order Riesz transforms acting on functions that are discontinuous across the moving interface. This makes it considerably more delicate to obtain Lipschitz bounds. Moreover, understanding the decay of the velocity is essential to exploit the smallness of the viscosity jump to close the estimates.
To handle the singular structure near the interface, we propagate regularity within each phase separately. Finally, uniqueness is established by switching to Lagrangian coordinates.

\subsection{Outline of the paper}
Section \ref{sect:Preliminaries} gathers some preliminaries, including notation, the definition of functional spaces, and auxiliary lemmas required for the following sections. Section \ref{sect:INS constant vis.} is devoted to the proof of Theorem \ref{Th1}, which is divided into existence and uniqueness. The existence part is organized into six steps. Theorem \ref{th2} is proved in Section \ref{sect:viscosity contrast}, also divided into existence and uniqueness, with the existence part structured in five steps.

\section{Preliminaries}\label{sect:Preliminaries}
This section is devoted to introducing the analytical framework and technical tools required for the proofs of the main results. We begin by defining the functional spaces involved in the analysis and fixing the notation. We then establish several interpolation results, specifically adapted to the setting considered here. Finally, we present a collection of auxiliary lemmas that will be used throughout the proofs of the main theorems.

\medskip
We denote the Lebesgue spaces by $L^p(\Omega)$, with $\Omega=\R^2$ or $\Omega=\T^2$, or simply $L^p$ when there is no risk of confusion. Given a normed functional  space $X$, such as a Sobolev or Hölder space, we denote the corresponding norm by $\norm{\c}_{X}$, except when $X=L^2$, in which case the subscript is omitted. For simplicity, the Bochner space $L^p(0,T; X)$ will be denoted by $L^p_T X$.

\medskip
We recall the definition of homogeneous and inhomogeneous Sobolev and Hölder spaces (see Chapter 1 in \cite{BCD}). 
    \begin{definition}
        Let $s\in\R$. The inhomogeneous Sobolev space $H^s(\R^2)$  consists of tempered distributions $u$ such that $\hat{u}\in L^2_{\mathrm{loc}}(\R^2)$ and 
        \begin{equation*}
            \norm{u}_{H^s}^2=\int_{\R^2} (1+\abs{\xi}^2)^s\abs{\hat{u}(\xi)}^2 d\xi < \infty.
        \end{equation*}
    \end{definition}
Analogously in the unit torus, $H^s(\T^2)$ is defined via the norm
\begin{equation*}
	\norm{u}_{H^s}^2=\sum_{k\in\Z^2}(1+\abs{k}^{2})^s\abs{\hat{u}_k}^2.
\end{equation*}
 \begin{definition}
        Let $s>0$. The homogeneous Sobolev space $\dot{H}^s(\R^2)$  consists of tempered distributions $u$ such that  $\hat{u}\in L^1_{\mathrm{loc}}(\R^2)$ and 
        \begin{equation*}
            \norm{u}_{\dot{H}^s}^2=\int_{\R^2} \abs{\xi}^{2s}\abs{\hat{u}(\xi)}^2 d\xi < \infty.
        \end{equation*}
 \end{definition}

 \begin{definition}
    Let $\Omega \subset \R^2$ and $\gamma\in (0,1)$. The inhomogeneous Hölder space $C^{1+\gamma}(\Omega)$ is the space of functions $u\in C^1(\Omega)$ such that 
    \begin{equation*}
        \norm{u}_{C^{1+\gamma}(\Omega)}=\norm{ u}_{L^\infty(\Omega)}+\norm{\nabla u}_{L^\infty(\Omega)}+\norm{\nabla u}_{\dot{C}^{\gamma}(\Omega)}< \infty,\quad\end{equation*}   
    where
    $$\norm{f}_{\dot{C}^{\gamma}(\Omega)}=\sup_{\substack{ x,y \in \Omega\\ x \neq y}}\frac{\abs{ f(x)- f(y)}}{\abs{x-y}^\gamma},$$
    denotes the Hölder seminorm.
 \end{definition}

We emphasize that the spaces $L^p$ and $H^s$ are Banach spaces for all $s\in\R$ and $p \in [1, \infty]$. In contrast, homogeneous Sobolev spaces are not always complete (see Proposition 1.34 and Remark 2.26 in \cite{BCD}). Most interpolation results in the literature apply to Banach spaces (see \cite{Bergh76, BCD}). However, we are interested in interpolating with the space $\dot{H}^1(\R^2)$, which is not complete. To deal with this, we rely on real interpolation, which can be formulated for general normed spaces and is denoted by $ (A, B)_{\theta, p}$.

\begin{proposition}(Theorem 2.6 in \cite{Gaudin24})\label{thfrances}
    Let $s_0,s_1 \in \R$ such that $s_0 \neq s_1$ and for $\th \in (0,1)$, let
    $$
    s=(1-\th)s_0+\th s_1.
    $$
    Assuming $s_0<1$, we get the following 
    $$
(\dot{H}^{s_0}(\R^2),\dot{H}^{s_1}(\R^2))_{\th,2}=\dot{H}^s(\R^2),
    $$
\end{proposition}
In particular, we will use the following interpolation in $\R^2$ for $s_0=0$ and $s_1=1$,
$$
(\dot{H}^{0}(\R^2),\dot{H}^{1}(\R^2))_{\th,2}=\dot{H}^{s}(\R^2),
$$
where by definition $\dot{H}^{0}(\R^2)=\Lp{2}$ and $s\in(0,1)$.
An analogous interpolation result holds for inhomogeneous spaces on $\T^2$ as well.

\medskip
Moreover, we would like to consider interpolation of $L^p$-spaces with change of measure, that is, we would like to study $\parent{L^p(dm_0),L^p(dm_1)}_{\th,q}$ where $m_0, m_1$ are two positive measures. We may assume that $m_0$ and $m_1$ are absolutely continuous with respect to a third measure $m$. Then, by Radon-Nikodym theorem there are two functions, $\om_0, \om_1$, such that
$$
\begin{array}{lr}
     dm_0(x)=\om_0(x) dm(x), &
     dm_1(x) = \om_1(x) dm(x). 
\end{array}
$$
We will denote by $L^p(\om_i)$ the space $L^p(\om_i d m)$ for $i=0,1$.

\begin{proposition}\label{Stein-Weiss} (Interpolation theorem of Stein-Weiss. Theorem 5.4.1 in \cite{Bergh76}). Assume that $p\in(0,\infty]$ and that $\th\in(0,1)$. Considering the function $\omega$ defined by
$$
\om(x)=\om_0^{1-\th}(x)\om_1^\th(x).
$$
Then
$$
\parent{L^p(\om_0),L^p(\om_1)}_{\th,p}=L^p(\om).
$$    
\end{proposition}

\subsection{Auxiliary lemmas for Theorem~\ref{Th1}}

We first introduce some lemmas satisfied by the  weak solutions of \eqref{VJ} in the whole space with $\mu=1$. The first three lemmas are necessary to obtain a bound of the $L^2$ norm of a function $f$ in the whole space by taking advantage of the control of $\rho f$ and its derivative. The proof of these lemmas can be found in \cite{Hao24a}. 

A first remark is the fact that $\rho_0\not\equiv 0$ implies the existence of $x_0\in \R^2$ and $R_0\in(0,\infty)$ such that 
\begin{equation}\label{lowrho}
  \int_{B(x_0,R_0)}\rho_0(x)dx\geq c_0 >0.  
\end{equation}
We would like to take advantage of the propagation in time of this property. 
\begin{lemma}(Lemma 2.2 in \cite{Hao24a})\label{lemma3.5}
    Assume \eqref{lowrho} and $R\geq R_0+2t\norm{\sr u_0}/\sqrt{c_0}$. Then we have 
    $$
    \int_{B(x_0,R)}\rho(t,x)dx\geq c_0/4. 
    $$
\end{lemma}

Next lemma shows local control of $f$ in terms of $\sqrt{\rho}f$ at the expense of higher regularity.

\begin{lemma}(Lemma A.2 in \cite{Hao24a})\label{Lemma3.6}
    Let $\varrho \in \Lp{\infty}$ and $\int_{B(x_0,r_0)}\varrho(x)dx\geq c_0 >0$. Then for all $p\in[1,\infty)$, there exist a constant $C>0$ depending on $p, \norm{\varrho}_{L^\infty}$ and $c_0$ such that for $f \in H^1_{\mathrm{loc}}(\R^2)$,
    $$
    \norm{f}_{L^p(B(x_0,r_0))}\leq C r_0^{2/p}\parent{(1+r_0)\norm{\nabla f}_{L^{2}(B(x_0,r_0))}+\norm{\sqrt{\varrho}f}_{L^2(B(x_0,r_0))}}.
    $$
\end{lemma}

The assumptions in $\rho_0$, \eqref{H1} or \eqref{H2},   allow a generalization of the previous lemma.

\begin{lemma}(Lemma 2.4 in \cite{Hao24a})\label{Lemma3.7}
    Let $p\in[2,\infty)$ and $T\in(0,\infty)$. If $\rho_0$ satisfies \eqref{H1}, then there exists a constant $C>0$ depending only on $T,p,\norm{\rho_0}_{L^\infty},\norm{\sr u_0}, \| (1+|x|^2)\rho_0\|_{L^1}$, such that for $f \in \dot{H}^1(\R^2)$,
    \begin{equation}
        C^{-1}\norm{\srho f}_{L^p} \leq \norm{f}_{L^2(B_1)}+\norm{\nabla f}\leq C (\norm{\srho f}+\norm{\nabla f})~~~ \forall t \in [0,T].
    \end{equation}
    If $\rho_0$ satisfies \eqref{H2}, then there exists a constant $C>0$ depending only on $T,p,R_0,c_0,\norm{\rho_0}_{L^\infty}$, $\norm{\sqrt{\rho_0} u_0}$ such that for $f \in \dot{H}^1(\R^2)$,
    \begin{equation}
        \norm{f}_{H^1}+\norm{\srho f}_{L^p}\leq C (\norm{\srho f}+\norm{\nabla f})~~~ \forall t \in [0,T].
    \end{equation}
\end{lemma}

%Global existence
In the global existence proof in Theorem \ref{Th1}, we need a version of Lemma 3.4 in \cite{Danchin19}.
\begin{lemma}\label{ThDanchin}
    Let $p\in[1,\infty]$ and $\apli{f}{(0,T)\times \Om}{\R}$ such that $f\in L^2_T L^p$ and $t^{1-\frac{\ga}{2}}\partial_t f\in L^2_T L^p$. Then, $f\in H^{\frac{1}{2}-\al}_T L^p$, $\forall \al \in (\frac{1-\ga}{2},\frac{1}{2})$, $\forall \ga \in (0,1)$. Moreover,
    \begin{equation*}
        \norm{f}^2_{H^{\frac{1}{2}-\al}_T L^p}\leq \norm{f}^2_{L^2_T L^p} + C \norm{t^{1-\frac{\ga}{2}}\partial_t f}^2_{L^2_T L^p},
    \end{equation*}
    where $C$ depends on $\al$, $\ga$ and $T$.
\end{lemma}

\begin{proof}
    By definition,
    \begin{equation}\label{3.16}
        \norm{f}^2_{H^{\frac{1}{2}-\al}_T L^p} = \norm{f}_{L^2_T L^p}^2 + \int_0^T\Big(\int_0^{T-h}\frac{\norm{f(t + h)-f(t)}^2_{L^p}}{h^{2-2\al}}dt\Big)dh. 
    \end{equation}
By the fundamental theorem of calculus, adding time weights followed by  Hölder's inequality, we have
\begin{equation*}
    \begin{split}
        \int_0^{T-h}\norm{f(t + h)-f(t)}^2_{L^p}dt & = \int_0^{T-h}\norm{\int_t^{t+h}\partial_t f(s)ds}^2_{L^p}dt \leq \norm{t^{\frac{2-\ga}{2}}\partial_t f}_{L^2_T L^p}^2\int_0^{T-h}\int_t^{t+h}\frac{1}{s^{2-\ga}}dsdt.
    \end{split}
\end{equation*}
Plugging the above inequality into \eqref{3.16},
\begin{equation*}
    \norm{f}^2_{H^{\frac{1}{2}-\al}_T L^p} \leq \norm{f}_{L^2_T L^p}^2 +  \norm{t^{\frac{2-\ga}{2}}\partial_t f}_{L^2_T L^p}^2 \int_0^T \frac{1}{h^{2-2\al}}\int_0^{T-h}\int_t^{t+h}\frac{1}{s^{2-\ga}}dsdtdh. 
\end{equation*}
It is easy to check that the above integral is finite for $\al \in \big(\frac{1-\ga}{2},\frac{1}{2}\big)$, which gives us the result.
\end{proof}

%Uniqueness
The uniqueness proof is based on a duality argument. The following auxiliary lemmas related to the estimation of a trilinear term will be needed.
\begin{lemma}(Lemma 4.1 in \cite{Hao24a})\label{tri1}
    There exists $C>0$ such that for any $f,g\in H^1$, $\varrho \in \dot{H}^{-1}\cap L^\infty$ with $\norm{\varrho}_{\dot{H}^{-1}}\leq 1/2$, and any $\ep \in (0,1/2]$ we have
    \begin{equation}
        \Big|\int_{\R^2}\varrho f gdx\Big|\leq C\parent{\norm{\varrho}_{\dot{H}^{-1}}\abs{\ln\ep}^{1/2}+\ep \norm{\varrho}_{L^\infty}}\norm{f}_{H^1}\norm{g}_{H^1}.
    \end{equation}
\end{lemma}
\begin{lemma}(Lemma 4.2 in \cite{Hao24a})\label{tri2}
    For any $C_0>0$, there exists $C>0$ such that for any $f,g\in H^1$, $\varrho \in \dot{H}^{-1}\cap L^\infty$ with $\norm{\varrho}_{\dot{H}^{-1}}\leq 1/2$, $\norm{\varrho}_{L^\infty}+\int_{\R^2}\abs{\varrho}(1+\abs{x}^2)dx\leq C_0$, for any $\ep\in(0,1/2]$ we have 
    \begin{equation}
        \Big|\int_{\R^2}\varrho f g dx\Big|\leq C\parent{\norm{\varrho}_{\dot{H}^{-1}}\abs{\ln\ep}^{1/2}+\ep}\parent{\norm{f}_{L^2(B_1)}+\norm{\nabla f}}\parent{\norm{g}_{L^2(B_1)}+\norm{\nabla g}}.
    \end{equation}
\end{lemma}

\subsection{Auxiliary lemmas for Theorem~\ref{th2}}
In the second part of the article, we work on the unit torus, which allows us to use the following Poincaré-type inequality instead of Lemma \ref{Lemma3.7}.

\begin{lemma}(Lemma A.1 in \cite{Danchin19})\label{PoincaréFull}
    Let $\apli{\varrho}{\T^2}{\R}$, $\varrho \geq0$ and $\varrho \not \equiv 0$. For every $f\in H^1(\T^2)$,
    \begin{equation*}
        \norm{f} \leq \frac{1}{M}\Big|\int_{\T^2}\varrho fdx\Big|+\Big(1+\frac{1}{M}\norm{M-\varrho }\Big)\norm{\nabla f},\quad\mbox{where}\quad M=\displaystyle \frac{1}{\abs{\T^2}}\int_{\T^2} \varrho dx.
\end{equation*}
\end{lemma}

Working on the torus allows us to assume that $\rho_0 u_0$ has zero mean. Indeed, if this is not the case, we can reduce to that situation by the change of variables
$\tilde{u}=u-C$, $\tilde{x}=x+Ct$, with $C=\parent{\int_{\T^2}\rho_0dx}^{-1}\int_{\T^2}\rho_0u_0dx$.
Therefore, since the mean of the momentum is preserved in time and $\rho D_t u$ necessarily has zero mean, the previous lemma can be simplified to
\begin{equation}\label{Poincaré}
       \norm{f}\leq C \norm{\rho}\norm{\nabla f},
\end{equation}
for $f=u$, $f=D_tu$. It should be pointed out that this inequality is the key point in obtaining the decay of the solution in Lemma \ref{ThDecay}.

%Linf gradiente
One of the main issues of Section \ref{sect:viscosity contrast} is the $L^1_T L^\infty$-norm of the velocity gradient. To achieve our goal, we make use of a version of Lemma 3.2 in \cite{Danchin19}, originally Lemma 2 in \cite{Desjardins97b}.
\begin{lemma}\label{LemDanchin2}
    There exists a constant $C>0$ such that for every $ f \in H^1(\T^2)$, $\varrho \in \LpT{\infty}$ with $0\leq \varrho(x)\leq \varrho^M$, mean value of $\varrho f$ zero and $2<p<\infty$, it holds
    \begin{equation*}
        \int_{\T^2} \varrho \abs{f}^pdx \leq C \norm{\sqrt{\varrho} f}^2\norm{\nabla f}^{p-2}\log^{\frac{p-2}{2}}\Big(e+\frac{\norm{\varrho-M}^2}{M^2}+\frac{\varrho^M\norm{\nabla f}^2}{\norm{\sqrt{\varrho} f}^2}\Big),
\,\,\mbox{where}\,\, M=\int_{\T^2}\varrho dx.   \end{equation*}
\end{lemma}

The uniqueness in this case will be achieved in Lagrangian variables as in \cite{Danchin19, Gancedo23}. We used the following version of Lemma A.2 in \cite{Danchin19}.

\begin{lemma}\label{Control_w}.
    Let $A$ be a matrix-valued function on $[0,T]\times \T^2$ satisfying $\det(A)=1$. There exists a constant c such that if 
    \begin{equation*}
        \norm{\I - A}_{L^\infty_T L^\infty}+\norm{\partial_tA}_{L^p_T L^q} < c,
    \end{equation*}
    with $p=\frac{2}{2-\ga}$ and $q\geq 2$, then for all functions $g \in L^2_t L^2$ satisfying
\begin{equation*}
    \begin{array}{lcr}
       g = \nabla \c R,  & R\in L^\infty_T L^2, & \partial_tR \in L^p_T L^r, 
    \end{array}
\end{equation*}    
    with $r=\frac{2q}{2+q}$, the equation $\div(Aw)=g$ in $[0,T]\times \T^2$ has a solution in the space
    $$
    W_T=\llav{w\in L^\infty_T L^2, \nabla w \in L^2_T L^2, \partial_t w \in L^p_T L^r},
    $$
    that satisfies
    $$
    \begin{array}{c}
\norm{w}_{L^\infty_T L^2}\leq C \norm{R}_{L^\infty_T L^2},\quad
\norm{\nabla w}_{L^2_T L^2} \leq C \norm{g}_{L^2_T L^2},\quad         \norm{\partial_t w}_{L^p_T L^r} \leq C \norm{R}_{L^\infty_T L^2} + C \norm{\partial_tR}_{L^p_T L^r}.
    \end{array}
    $$
\end{lemma}

\section{Proof of Theorem \ref{Th1}}\label{sect:INS constant vis.}
 \subsection{Existence.} 
 Existence is established through a strategy that combines mollification with compactness methods. After smoothing the initial data, we establish a priori estimates for the unique smooth solution. A key component is an estimate showing that the velocity is $L^1$-in-time Lipschitz-in-space, which guarantees global-in-time existence. Since these estimates are uniform with respect to the mollifier parameter, we can use compactness to pass to the limit (see Step 6 below). We now move on to derive the a priori estimates.

 \medskip
 \textbf{Step 1: }$\srho u \in L^\infty_T L^2$ and $\nabla u \in L^2_T L^2$.

\medskip
    Testing the momentum equation of \eqref{VJ} with $u$ we obtain 
    \begin{equation*}
        \produ{\rho \partial_t u}{u}+\produ{\rho \matt{u}}{u}= \produ{\De u}{u} + \produ{\nabla P}{u}.
    \end{equation*}
    Using that $\div u=0$, the equation for $\rho$ and integration by parts we get
    \begin{equation*}
            \frac{1}{2}\frac{d}{dt}\norm{\srho u}^2=-\norm{\nabla u}^2.
    \end{equation*}
    After integrating in time we obtain the classical weak energy inequality \eqref{energy_ineq} for $\mu=1$.

\medskip
\textbf{Step 2: }  $t^{\frac{1-(\ga+s)}{2}}\sqrt{\rho}D_t u\in L^2_T L^2$ and 
        $t^{\frac{1-(\ga+s)}{2}}\nabla{u}\in L^\infty_T L^2$.

\medskip
We proceed with an interpolation argument and time-weighted estimates. 
Consider the linearized problem 
\begin{equation*}
    \left\{
    \begin{aligned}
       \mat{\rho}&=0, \\
       \rho\parent{\mat{v}}&=\De v + \nabla P, \\
       \nabla \c v &= 0, \\
       \parent{\rho,v}|_{t=0}&=\parent{\rho_0,v_0}.
 \end{aligned}
\right.
\end{equation*}
We test the momentum equation against $\dv=\mat{v}$,
    \begin{equation}\label{2.2}
        \produ{\rho\dv}{\dv}= \produ{\De v}{\dv} + \produ{\nabla P}{\dv}.
    \end{equation}
The second term in the right-hand side is bounded as follows
    \begin{equation*}
        \begin{split}
            |\produ{\nabla P}{\dv}|&\leq\abs{\produ{\nabla P}{\partial_t v}}+\abs{\produ{\nabla P}{\matt{v}}}
            = \Big|\int_{\R^2} P \pr_i(u_j\pr_jv_i) dx\Big|= \Big|\int_{\R^2} P \pr_iu_j\pr_jv_i dx\Big|,
        \end{split}
    \end{equation*}
    where we have used the incompressibility condition and integrated by parts. We can apply the Div-Curl Lemma (Theorem II.1 in \cite{Coifman93}), as $\div{\pr_i u}=0$ and $\curl (\nabla v_i)=0$, and Sobolev embeddings to obtain
\begin{equation}\label{divcurl}
    \Big|\int_{\R^2} P \pr_iu_j\pr_jv_i dx\Big| \leq C \norm{P}_{BMO} \norm{\pr_iu_j\pr_jv_i}_{\HH^1}\leq C \norm{\nabla P}\norm{\nabla u}\norm{\nabla v}.
\end{equation}

Next, using $\div{v}=0$ again, integration by parts and the Hölder's inequality in the first term of the right-hand side of \eqref{2.2} it is possible to get
\begin{equation*}
    \begin{split}
    \produ{\De v}{\dv}&=\produ{\De v}{\partial_t v}+\produ{\De v}{\matt{v}} = -\frac{1}{2}\frac{d}{dt}\norm{\nabla v}^2 -\int_{\R^2} \partial_j v_j \partial_j u_k \partial_k v_j dx - \int_{\R^2} \partial_j v_j  u_k \partial_k \partial_j v_j dx\\
    &\leq -\frac{1}{2}\frac{d}{dt}\norm{\nabla v}^2 + C \norm{\nabla u}\norm{\nabla v}^{2}_{L^4} \leq -\frac{1}{2}\frac{d}{dt}\norm{\nabla v}^2 + C \norm{\nabla u}\norm{\nabla v}\norm{\nabla^2 v}.
    \end{split}
\end{equation*}
Joining the above bounds in \eqref{2.2} we get
\begin{equation*}
\begin{split}
    \norm{\srho \dv}^2+\frac{1}{2}\frac{d}{dt}\norm{\nabla v}^2 &\leq C\norm{\nabla u}\norm{\nabla v}\norm{\nabla^2 v} + C\norm{\nabla P}\norm{\nabla u}\norm{\nabla v}\\
    \leq \frac{1}{2}&\parent{\norm{\nabla^2 v}^2+\norm{\nabla P}^2}+C\norm{\nabla u}^2\norm{\nabla v}^2
    \leq \frac{1}{2}\norm{\srho \dv}^2+C\norm{\nabla u}^2\norm{\nabla v}^2,
    \end{split}
\end{equation*}
where we also used the Stokes equation and Young inequality. 
The above inequality can be rewritten as
\begin{equation}\label{in2}
    \norm{\srho \dv}^2+\frac{d}{dt}\norm{\nabla v}^2 \leq C \norm{\nabla u}^2\norm{\nabla v}^2.
\end{equation}

Introducing a time weight it is possible to get 
\begin{equation*}
    t\norm{\srho \dv}^2+\frac{d}{dt}(t\norm{\nabla v}^2) \leq C t \norm{\nabla u}^2\norm{\nabla v}^2 +\norm{\nabla v}^2.
\end{equation*}
Using Gr\"onwall's Lemma and \eqref{energy_ineq} for $v$ gives
\begin{equation}\label{4.5}
\begin{split}
    \int_{0}^{t}\tau\norm{\srho \dv}^2d\tau+ &t\norm{\nabla v}^2 \leq C \exp{\Big(\int_{0}^{t}\norm{\nabla u}^2d\tau}\Big)\int_{0}^{t}\norm{\nabla v}^2d\tau \leq C \norm{\rho_0 v_0}^2\leq C \norm{v_0}^2.
\end{split}
\end{equation}
Supposing that $v_0\in \dot{H}^1(\R^2)$ in \eqref{in2} we can apply Gr\"onwall's Lemma without introducing time weights to obtain
\begin{equation}\label{4.6}
    \int_{0}^{t}\norm{\srho \dv}^2 d\tau+ \norm{\nabla v}^2 \leq C \exp{\Big(\int_{0}^{t}\norm{\nabla u}^2 d\tau}\Big) \norm{\nabla v_0}^2\leq C \norm{\nabla v_0}^2.
\end{equation}
Interpolation between the two inequalities is carried out using an argument similar to \cite{Paicu13,Gancedo18}. We define the linear operator $L v_0=\nabla v$. The inequalities \eqref{4.5} and \eqref{4.6} give us
\begin{equation*}
   \begin{array}{lr}
        \norm{Lv_0}\leq C \norm{v_0}_{\dot{H}^1},  &
        \norm{Lv_0}\leq C t^{-1/2} \norm{v_0}.
   \end{array} 
\end{equation*}
Using Proposition \ref{thfrances}, we obtain
\begin{equation*}
    \norm{Lv_0}\leq C t^{-\frac{1-(\gamma+s)}{2}} \norm{v_0}_{\Hgam}.
\end{equation*}
Similarly, we can define the linear operator $Lv_0=\srho \dv$. It follows from \eqref{4.5} and \eqref{4.6} that
\begin{equation*}
    \begin{array}{lr}
          \displaystyle \int_{0}^{t}\norm{L v_0}^2 d\tau  \leq C \norm{v_0}^2_{\dot{H}^1}, &
        \displaystyle \int_{0}^{t}\tau\norm{L v_0}^2 d\tau \leq C  \norm{v_0}^2.
    \end{array}
\end{equation*}
Hence, it satisfies
\begin{equation*}
    \begin{array}{lr}
         \apli{L}{\Lp{2}}{L^2\parent{\parent{\corch{0,T},tdt};\Lp{2}}},&
        \apli{L}{\dot{H}^1\parent{\R^2}}
        {L^2\parent{\parent{\corch{0,T},dt};\Lp{2}}}.
    \end{array}
\end{equation*}
Applying Proposition \ref{thfrances} we obtain
$$(\Lp{2},\dot{H}^{1}(\R^2))_{\th,2}=\Hgam.$$
Also, since both measures are absolutely continuous with respect to the Lebesgue measure in $[0,T]$, it follows by Proposition \ref{Stein-Weiss} that
$$\parent{L^2\parent{\parent{\corch{0,T},tdt};\Lp{2}},L^2\parent{\parent{\corch{0,T},dt};\Lp{2}}}_{\th,2}=L^2((\corch{0,T},t^{1-(\ga+s)}dt);\Lp{2}).$$
Therefore, we can conclude  that 
$$
\apli{L}{\Hgam}
        {L^2((\corch{0,T},t^{1-(\gamma+s)}dt};\Lp{2})).
$$
Putting together both interpolations and making $v=u$, we get the following inequality, 
\begin{equation}\label{Lema3.3}
\begin{split}
    \int_{0}^{t}\tau^{1-(\gamma+s)}\norm{\srho \du}^2 d\tau+ t^{1-(\ga+s)}\norm{\nabla u}^2 &\leq C\norm{u_0}^2_{\dot{H}^{\ga+s}}.
\end{split}
\end{equation}

\medskip
\textbf{Step 3: } $t^{\frac{2-(\ga+s)}{2}}\sqrt{\rho}D_t u\in L^\infty_T L^2$ and $t^{\frac{2-(\ga+s)}{2}}\nabla \du \in L^2_TL^{2}$.

\medskip
  We take $D_t$ in the momentum equation of \eqref{VJ},
    $$
    \rho(\mat{\du}) =\Delta \du + \partial_k(\partial_k u \c \nabla u)-\partial_k u \cdot \nabla(\partial_k u)-D_t \nabla P.
    $$

    Testing the above equation with $\du$, integration by parts and taking into account that $\div{u}=0$, it follows that
    \begin{equation}\label{est3}
        \begin{split}
            \frac{1}{2}\frac{d}{dt}\norm{\srho \du}^2+\norm{\nabla \du}^2
            =&-\int_{\R^2}(\partial_k u \cdot  \nabla u )\cdot \partial_k \du dx +\int_{\R^2}\partial_k u \cdot (\partial_k u \c \nabla\du) dx \\
            & -\int_{\R^2}(\nabla \partial_tP+u\c\nabla(\nabla P))\c\du dx =J_1+J_2+J_3.
        \end{split}
    \end{equation}

To estimate $J_1$ and $J_2$, we combine Hölder's inequality, the Gagliardo-Nirenberg inequality, Young inequality, and properties of the Stokes operator to obtain the following
\begin{equation*}
    J_1+J_2\leq C \norm{\nabla \du}\norm{\nabla u}_{L^4}^2 \leq C \norm{\nabla \du}\norm{\nabla u}\norm{\nabla^2u}\leq \frac{1}{8} \norm{\nabla \du}^2 + C\norm{\srho\du}^2\norm{\nabla u}^2.
\end{equation*}

On the other hand, integrating by parts, $J_3$ can be rewritten as
\begin{equation}\label{J_3}
    \begin{split}
        J_3&= \int_{\R^2}\partial_tP \nabla \c {\du} dx + \int_{\R^2}(\matt{P})\div{\du} dx +\int_{\R^2}\partial_i u_j \partial_j P \du_i dx.
    \end{split}
\end{equation}
Now, as $\div{u}=0$ and $\div{\du}=\partial_ju_i\partial_iu_j$, we have
$$
J_3=\int_{\R^2}(\partial_tP+\matt{P})\partial_iu_j\partial_ju_i dx-\int_{\R^2}P\partial_iu_j\partial_j\du_i dx =J_{3,1}-J_{3,2}.
$$

We can rewrite $J_{3,1}$ as follows,
\begin{equation*}
    J_{3,1}=\frac{d}{dt}\int_{\R^2}P\partial_iu_j\partial_ju_i dx-\int_{\R^2}PD_t(\partial_iu_j\partial_ju_i)dx = \frac{d}{dt}\int_{\R^2}P\partial_iu_j\partial_ju_i dx - 2J_{3,2}.
\end{equation*}

where we have used the fact that if $A$ is a $2\times 2$ matrix with $\tr(A)=0$ then $\tr(A^3)=0$, i.e $\partial_iu_k\partial_ku_j\partial_ju_i=0$.

Applying the same reasoning used in \eqref{divcurl}, we derive the following bound for $J_{3,2}$
\begin{equation*}
    J_{3,2}\leq C \norm{P}_{BMO}\norm{\partial_i u_j \partial_j \du_i}_{\HH^1}\leq C \norm{\nabla P}\norm{\nabla u}\norm{\nabla \du} \leq \frac{1}{8}\norm{\nabla \du}^2 + C \norm{\srho \du}^2\norm{\nabla u}^2.
\end{equation*}
Combining the bounds for $J_1$, $J_2$ and $J_3$ in \eqref{est3}, we arrive at
\begin{equation*}
    \begin{split}
        \frac{1}{2}\frac{d}{dt}\norm{\srho \du}^2+\norm{\nabla \du}^2&\leq\frac{1}{2} \norm{\nabla \du}^2 + \frac{d}{dt}\int_{\R^2}P\partial_iu_j\partial_ju_i dx + C \norm{\srho \du}^2\norm{\nabla u}^2,
    \end{split}
\end{equation*}
which can be rewritten as
\begin{equation}\label{est3ph}
    \frac{d}{dt}\Big(\frac{1}{2}\norm{\srho \du}^2-\int_{\R^2}P\partial_iu_j\partial_ju_i dx\Big)+\frac{1}{2}\norm{\nabla \du}^2\leq C \norm{\srho \du}^2\norm{\nabla u}^2.
\end{equation}

Denote $\displaystyle \ph(t)=\frac{1}{2}\norm{\srho \du}^2-\int_{\R^2}P\partial_iu_j\partial_ju_i dx$. We derive both upper and lower bounds for this function. Proceeding as in \eqref{divcurl}, we obtain
\begin{equation*}
    \Big|\int_{\R^2} P \pr_iu_j\pr_ju_i dx\Big| \leq \norm{\nabla P}\norm{\nabla u}^2\leq \frac{1}{4}\norm{\srho\du}^2+C\norm{\nabla u}^4.
\end{equation*}
Then, the function $\ph(t)$ can be bounded as follows
\begin{equation}\label{ULBound}
    \frac{1}{4}\norm{\srho\du}^2-C\norm{\nabla u}^4 \leq \ph(t) \leq \frac{3}{4}\norm{\srho\du}^2+C\norm{\nabla u}^4.
\end{equation}

Introducing the time weight $t^{2-(\ga+s)}$ in \eqref{est3ph} and using the upper bound of \eqref{ULBound} yields
\begin{equation*}
\begin{split}
    \frac{d}{dt}&(t^{2-(\ga+s)}\ph(t))+\frac{t^{2-(\ga+s)}}{2}\norm{\nabla \du}^2\leq C t^{2-(\ga+s)}\norm{\srho \du}^2\norm{\nabla u}^2+Ct^{1-(\ga+s)}\ph(t) 
    \\
    &\leq C t^{2-(\ga+s)}\norm{\nabla u}^2\ph(t)+C t^{2-(\ga+s)}\norm{\nabla u}^6+ C t^{1-(\ga+s)}\norm{\srho\du}^2+C t^{1-(\ga+s)}\norm{\nabla u}^4.
\end{split}
\end{equation*}

Gr\"onwall's Lemma, \eqref{energy_ineq} and  \eqref{Lema3.3} gives
\begin{equation*}
\begin{split}
    &t^{2-(\ga+s)}\ph(t)+\frac{1}{2}\int_{0}^{t}\tau^{2-(\ga+s)}\norm{\nabla \du}^2 d\tau\\
    &\leq C \exp{\Big(\int_0^t\norm{\nabla u}^2 d\tau}\Big)
    \corch{\int_0^t \tau^{2-(\ga+s)}\norm{\nabla u}^6 dt+\int_0^t \tau^{1-(\ga+s)}(\norm{\srho\du}^2 + \norm{\nabla u}^4)} d\tau \leq C.
    \end{split}
\end{equation*}

To conclude, we use the lower bound of \eqref{ULBound} to obtain
\begin{equation}\label{Lema3.4}
    t^{2-(\ga+s)}\norm{\srho\du}^2+\int_{0}^{t}\tau^{2-(\ga+s)}\norm{\nabla \du}^2 d\tau \leq C(\norm{\rho_0 u_0},\norm{u_0}_{\dot{H}^{\ga+s}},T) .
\end{equation}

\medskip
\textbf{Step 4: } $\nabla u \in L^1_T L^\infty$.

\medskip
In order to control the velocity gradient, we do need the structure of the density. We first introduce a lemma that allows us to bound $\nabla^2 u$.
\begin{lemma}\label{Lemma3.8}
    If $\rho_0$ satisfies \eqref{H1} or \eqref{H2} then for all $m\in[2,\infty)$, $T>0$, it holds that
$$t^{\frac{2-(\ga+s)}{2}}\nabla^2u\in L^2(0,T;\Lp{m}).$$
\end{lemma}
\begin{proof} 
Applying Stokes, it follows that
    \begin{equation*}
    \begin{split}
        \norm{t^{\frac{2-(\ga+s)}{2}}\nabla^2u}&_{L^2_TL^m}  \leq \norm{t^{\frac{2-(\ga+s)}{2}}\srho \du}_{L^2_TL^m}
        \\
        & \leq C \parent{\norm{t^{\frac{2-(\ga+s)}{2}}\srho \du}_{L^2_TL^2}+\norm{t^{\frac{2-(\ga+s)}{2}}\nabla \du}_{L^2_TL^2}}\leq C(\norm{\rho_0 u_0},\norm{u_0}_{\dot{H}^{\ga+s}},T),
        \end{split}
    \end{equation*}
    where Lemma \ref{Lemma3.7} has also been applied, noting that $\du\in\dot{H}^1$ as established in \eqref{Lema3.4}.
\end{proof}

Now we can prove that $\nabla u \in L^1_T L^\infty$. The $L^\infty$-Gagliardo–Nirenberg inequality (2.3.50 in \cite{Mazya2011}) with $\th=(2(1-1/m))^{-1}\in(0,1)$ and $m>2$, yields
    \begin{equation*}
        \int_0^T \norm{\nabla u}_{L^\infty} d\tau \leq C \int_0^T \norm{\nabla u}^{1-\th}\norm{\nabla^2 u}^\th_{L^m} dt.
    \end{equation*}
Applying \eqref{Lema3.3}, Hölder's inequality for $p=2/\th$ and Lemma \ref{Lemma3.8}, it follows that
\begin{equation*}
    \begin{split}
        \int_0^T \norm{\nabla u}_{L^\infty} dt &  \leq C \int_0^T t^{-\frac{(1-(\ga+s))(1-\th)}{2}}\norm{\nabla^2 u}^\th_{L^m} dt
        \\
        &\leq C \Big(\int_0^T t^{-\frac{1+\th-(\ga+s)}{2-\th}}dt\Big)^{\frac{2-\th}{2}}\leq C(\norm{\rho_0 u_0},\norm{u_0}_{\dot{H}^{\ga+s}},T),
    \end{split}
\end{equation*}
which is bounded if $m>1+(\ga+s)^{-1}$.

\medskip
\textbf{Step 5: }$\nabla u \in L^1_T \Cga$.

\medskip
For $r = \ga+2/p$, $p<\infty$, the embedding $ \dot{W}^{r,p}(\R^2) \hookrightarrow \Cga(\R^2)$ holds,  from which we deduce the following,
\begin{equation*}
        \int_0^T \norm{\nabla u}_{\Cga}dt\leq C \int_0^T \norm{\nabla u}_{\dot{W}^{r,p}}dt \leq C \int_0^T \norm{\nabla u}^{1-\th}\norm{\nabla^2 u}^\th_{L^m}dt,
\end{equation*}
  where the Gagliardo-Nirenberg inequality (see \cite{Mazya2011})  is used in the second inequality for $m>p>2/(1-\gamma)$, with $\th=\frac{1+\ga}{2(1-1/m)}\in(0,1)$.
  As before, the use of \eqref{Lema3.3}, Hölder's inequality with exponent $2/\th$, and Lemma \ref{Lemma3.8} yields
    \begin{equation*}
        \int_0^T \norm{\nabla u}_{\Cga} dt \leq C \Big(\int_0^T t^{-\frac{1+\th-(\ga+s)}{2-\th}}dt\Big)^{\frac{2-\th}{2}}\leq C(\norm{\rho_0 u_0},\norm{u_0}_{\dot{H}^{\ga+s}},T),
    \end{equation*}
  which is bounded if $m>1+(1+\ga)/s$. 
  
\medskip
\textbf{Step 6:} We consider the smoothed-out approximate data with no vacuum to take advantage of the classical results. We define
\begin{equation*}
    u_0^\ep\in C^\infty(\R^2) \text{ with } \div{u_0^\ep}=0 \text{ and } \rho_0^\ep\in C^\infty(\R^2) \text{ with } 0<\ep \leq \rho_0^\ep \leq C_0
\end{equation*}
such that 
\begin{equation*}
             u_0^\ep \to u_0 \text{ in } \dot{H}^{\ga+s},\quad
         \rho_0^\ep \to \rho_0 \text{ in } L^\infty\text{-weak}^*,\quad
         \rho_0^\ep \to \rho_0 \text{ in } L^p_{\mathrm{loc}} \text{ if } p<\infty.
\end{equation*}

Since $\rho_0\not \equiv 0$, there exist $R_0\in(0,\infty)$ and $c_0>0$ such that $\int_{B(0,R_0)}\rho_0dx>c_0>0$. It follows that, for $\ep$ sufficiently small, $\int_{B(0,R_0)}\rho_0^\ep dx>c_0>0$ . In other words, \eqref{lowrho} holds with $x_0=0$ and $\rho_0$ replaced by $\rho_0^\ep$. Moreover, we have the bound $\norm{ \sqrt{\rho_0^\ep} u_0^\ep}\leq\norm{\sqrt{\rho_0}u_0}$. The classical theory for strong solutions to system \eqref{VJ} guarantees the existence of a unique global smooth solution $\parent{\rho^\ep, u^\ep}$ corresponding to the initial data $\parent{\rho_0^\ep,u_0^\ep}$. This solution satisfies the a priori estimates \eqref{energy_ineq}, \eqref{Lema3.3} and \eqref{Lema3.4}.

We aim to establish the following bound
\begin{equation*}
    \norm{u^\ep}^2_{H^{\frac{1}{2}-\al}L^2(B_R)}\leq \norm{u^\ep}^2_{L^2_T L^2(B_R)} + C \norm{t^{1-\frac{\ga+s}{2}}\partial_t u^\ep}^2_{L^2_T L^2(B_R)},
\end{equation*}
in other words, to have Lemma \ref{ThDanchin} for $f=u^\ep$. We proceed by obtaining a bound for $\partial_t u^\ep$.

Let $R\geq R_0+2T\norm{\sqrt{\rho_0^\ep}u_0^\ep}/\sqrt{c_0}$, by Lemma \ref{lemma3.5} we have $\int_{B(0,R)}\rho^\ep(x,t)dt\geq c_0/4$ and we can apply Lemma \ref{Lemma3.6} to $B_R=B(0,R)$. Now we can prove the bound for $\partial_t u^\ep$ in $B_R$.
Rewriting $\partial_t u^\ep =\du^\ep-u^\ep\cdot\nabla u^\ep$, it follows
\begin{equation*}
    \begin{split}
        \int_0^T& t^{2-(\ga+s)}\norm{\partial_t u^\ep}^2_{L^2(B_R)} dt  \leq \int_0^T t^{2-(\ga+s)}\norm{\du^\ep}_{L^2(B_R)}^2dt+\int_0^Tt^{2-(\ga+s)}\norm{u^\ep \c \nabla u^\ep}_{L^2(B_R)}^2 dt
        \\
        & \leq C(R) \int_0^T t^{2-(\ga+s)}\parent{\norm{\sqrt{\rho^\ep} \du^\ep}^2 + \norm{\nabla \du^\ep}^2} dt+ \int_0^T t^{2-(\ga+s)} \norm{u^\ep \c \nabla u^\ep}_{L^2(B_R)}^2 dt \\
        & \leq C(R) +C(R)\int_0^T t^{2-(\ga+s)} \norm{u^\ep}_{L^4(B_R)}^2\norm{\nabla u^\ep}_{L^4(B_R)}^2 dt.
\end{split}
\end{equation*}
where we have used Lemma \ref{Lemma3.6} in the second inequality and the a priori estimates for $u^\ep$ and Hölder's inequality in the third one.

Applying Gagliardo-Nirenberg inequality and Lemma \ref{Lemma3.6} it follows  
\begin{equation*}
    \begin{split}
        \int_0^Tt^{2-(\ga+s)}&\norm{\partial_tu^\ep}^2_{L^2(B_R)} dt \leq C(R)\Big(1+ \int_0^T t^{2-(\ga+s)} \norm{u^\ep}_{L^2(B_R)}\norm{\nabla u^\ep}^2\norm{\nabla^2 u^\ep} dt\Big)
        \\
        &\leq C(R)\Big(1+ \int_0^T t \norm{\rho^\ep u^\ep}\norm{\nabla^2 u^\ep} dt +\int_0^Tt\norm{\nabla u^\ep}\norm{\nabla^2 u^\ep}dt\Big)
        \\
        &\leq C(R)\Big(1+ \int_0^T t^{\frac{\ga+s}{2}}t^{\frac{2-(\ga+s)}{2}}\norm{\nabla^2 u^\ep}dt+\int_0^T t^{-\frac{1}{2}+(\ga+s)}t^{\frac{2-(\ga+s)}{2}}\norm{\nabla^2 u^\ep}dt\Big)
        \leq C(R),
    \end{split} 
\end{equation*}
where in the last steps we  used the a priori estimates for $u^\ep$. Then, we can conclude that for all $R \geq R_0 + 2 T \norm{\sqrt{\rho_0^\ep}u_0^\ep}/\sqrt{c_0}$,
$$
\int_0^T t^{2-(\ga+s)}\norm{\partial_t u^\ep}^2_{L^2(B_R)} dt  \leq C(R).
$$
However, the left-hand side is increasing in $R$ so we have that, for all $R>0$,
$$
\int_0^T t^{2-(\ga+s)}\norm{\partial_t u^\ep}^2_{L^2(B_R)} dt  \leq C(R).
$$
Therefore, by Lemma \ref{ThDanchin}, for all $\al \in (\frac{1-(\ga+s)}{2},\frac{1}{2})$
\begin{equation*}
    \norm{u^\ep}^2_{H^{\frac{1}{2}-\al}L^2(B_R)}\leq \norm{u^\ep}^2_{L^2_T L^2(B_R)} + C \norm{t^{1-\frac{\ga+s}{2}}\partial_t u^\ep}^2_{L^2_T L^2(B_R)}.
\end{equation*}

The a priori estimates yields $\norm{\nabla u^\ep}_{L^2_T L^2(B_R)}\leq C$ from which we deduce that $\norm{u^\ep}_{L^2_T H^1(B_R)}\leq C(R)$. 
By interpolation, it follows that $\norm{u^\ep}_{H^{\frac{1}{2}-\al}([0,T]\times B_R)}\leq C(R)$. Therefore, there exists a subsequence of $(u^\ep)$ (which we denote again by $(u^\ep)$ for simplicity ) and a limit function $u\in L^2_{\mathrm{loc}}([0,T]\times \R^2)$ such that
$$
u^\ep\to u \text{ in } L^2_{\mathrm{loc}}([0,T]\times \R^2).
$$
Moreover, we have that
$$
\nabla u^\ep  \rightharpoonup \nabla u \text{ in } L^2_{\mathrm{loc}}([0,T]\times \R^2).
$$
Theorem 2.5 in \cite{Lions} yields, for any $p\in[1,\infty)$,
$$
\rho^\ep\to\rho \text{ in } C([0,T]; L^p(B_R)).
$$

It follows that the equation holds in the strong sense in $L^2_{\mathrm{loc}}([0,T]\times \R^2)$ and that the limit solution $(\rho,u)$ satisfies the same a priori estimates.

\subsection{Uniqueness}
Assume that $(\rho,u)$ is a weak solution of \eqref{VJ}  and $(\bar{\rho},\bar{u})$ a strong one with the same initial data $(\rho_0,u_0)$. Let us denote $P$, $\bar{P}$ the corresponding pressures. We focus on $T\in(0,1/2)$ sufficiently small. We define $\delta \rho = \rho - \bar{\rho}$, $\delta u=u-\bar{u}$ and $\delta P = P-\bar{P}$  such that they satisfy the following system.

 \begin{equation}\label{UNI}
    \left\{
    \begin{aligned}
       \partial_t \delta\rho+\bar{u}\cdot \nabla \delta \rho + \delta u \cdot \nabla \rho&=0, \\
       \rho\parent{\mat{\delta u}}-\De \delta u + \nabla\delta P&=-\delta\rho D_t \bar{u}-\rho \delta u \cdot \nabla \bar{u}, \\
       \nabla \c\delta u &= 0, \\
       \parent{\delta\rho,\delta u}|_{t=0}&=\parent{0,0},
 \end{aligned}
\right.
\end{equation}

where $D_t \bar{u}=\partial_t \bar{u}+\bar{u}\c \nabla \bar{u}$.

We set $\phi := (-\Delta)^{-1}\delta \rho$ so that $\norm{\nabla \phi}=\norm{\delta \rho}_{\dot{H}^{-1}}$. 
Testing the equation of $\delta \rho$ with $\phi$, integrating by parts and using the divergence free condition for $\bar{u}$, we get
\begin{equation*}
    \frac{d}{dt}\norm{\nabla \phi}^2\leq C \norm{\nabla \bar{u}}_{L^\infty}\norm{\nabla \phi}^2 + C \norm{\rho \delta u}\norm{\nabla \phi}.
\end{equation*}
Applying Gr\"onwall's Lemma and Hölder's inequality it follows that 
\begin{equation}\label{preD}
    \norm{\delta \rho(t)}_{\dot{H}^{-1}}\leq C\exp{\Big(\int_0^T\norm{\nabla \bar{u} }_{L^\infty} dt\Big)}\norm{\srho \delta u}_{L^1_t L^2}\leq Ct^{1/2} \norm{\srho \delta u}_{L^2_t L^2} .
\end{equation}

In order to control $\srho \delta u$ in $L^2_t L^2$ we define the following linear backward parabolic system 
\begin{equation}\label{duality}
\left\{
    \begin{aligned}
       \rho\parent{\mat{v}}+\De v + \nabla  Q &=\rho \delta u, \\
       \nabla \c v &= 0, \\
       v|_{t=T}&=0.
 \end{aligned}
\right.
\end{equation}
By testing the momentum equation with \( v \) and \( D_t v \), and applying estimates similar to those used for the derivation of the a priori estimates, we are able to obtain the following
\begin{equation}\label{boundsv}
    \norm{\srho v, \nabla v}^2_{L^\infty_T L^2}+\int_0^T \parent{\norm{\nabla v}^2 + \norm{\srho D_t v, \nabla Q, \nabla^2 v}^2}dt \leq C \norm{\srho \delta u}_{L^2_T L^2}^2.
\end{equation}

Testing the momentum equation of $\delta u$ with $v$ yields
\begin{equation}\label{cota}
     \norm{\srho \delta u}_{L^2_T L^2}^2 \leq \int_0^T \abs{\produ{\delta \rho D_t \bar{u}}{v}}dt + \int_0^T\abs{\produ{\rho \delta u \c \nabla \bar{u}}{v}}dt=J_4+J_5.
\end{equation}

Hölder's inequality with $p=2+\varepsilon$ and $p'=\frac{2(2+\varepsilon)}{\varepsilon}$, $\varepsilon>0$, Lemma \ref{Lemma3.7}, \eqref{boundsv} and Gagliardo-Nirenberg inequality gives
\begin{equation*}
\begin{split}
    J_5 &\leq C \int_0^T \norm{\srho \delta u} \norm{\srho v}_{L^{p'}}\norm{\nabla \bar{u}}_{L^p}dt
    \leq C\int_0^T \norm{\srho \delta u} \parent{\norm{\srho v}+\norm{\nabla v}}\norm{\nabla \bar{u}}^{1-\th}\norm{\nabla^2 \bar{u}}^\th dt
    \\
    &\leq C\norm{\srho \delta u}_{L^2_T L^2}^2 \Big(\int_0^T \norm{\nabla \bar{u}}^{2(1-\th)}\norm{\nabla^2 \bar{u}}^{2\th}dt\Big)^{1/2} \leq C(T) \norm{\srho \delta u}_{L^2_T L^2}^2,
\end{split}
\end{equation*}

where $\th=\frac{\varepsilon}{2+\varepsilon}$ and we chose $\varepsilon<\frac{2(\ga+s)}{1-(\ga+s)}$. Adjusting $T$, $C(T)$ can be small enough and \eqref{cota} becomes
\begin{equation}\label{cota2}
    \norm{\srho \delta u}_{L^2_T L^2}^2 \leq C \int_0^T \abs{\produ{\delta \rho D_t \bar{u}}{v}}dt = C J_4.
\end{equation}
We can take $\ep=\norm{\delta \rho}_{\dot{H}^{-1}}<1/2$ small by \eqref{preD} in Lemma \ref{tri1} and Lemma \ref{tri2} and use Lemma \ref{Lemma3.7} with \eqref{boundsv} to obtain
\begin{equation*}
\begin{split}
    J_4 &\leq C \int_0^T \norm{\delta \rho}_{\dot{H}^{-1}} \abs{\ln{(\norm{\delta \rho}_{\dot{H}^{-1}})}}^{1/2} \parent{\norm{\srho D_t \bar{u}}+\norm{\nabla D_t \bar{u}}}\parent{\norm{\srho v}+\norm{\nabla v}}dt
    \\
    & \leq C \norm{\srho \delta u}_{L^2_T L^2} \int_0^T \norm{\delta \rho}_{\dot{H}^{-1}} \abs{\ln{(\norm{\delta \rho}_{\dot{H}^{-1}})}}^{1/2} \parent{\norm{\srho D_t \bar{u}}+\norm{\nabla D_t \bar{u}}}dt.
\end{split}
\end{equation*}

Introducing the above estimate in \eqref{cota2} and using the bound \eqref{preD}
\begin{equation*}
    \norm{\srho \delta u}_{L^2_T L^2} \leq C \int_0^T \norm{\srho \delta u}_{L^2_t L^2} \abs{\ln{(t^{1/2}\norm{\srho \delta u}_{L^2_t L^2})}}^{1/2} t^{1/2}\parent{\norm{\srho D_t \bar{u}}+\norm{\nabla D_t \bar{u}}}dt.
\end{equation*}
Note that we can take $t$ such that $\norm{\srho \delta u}_{L^2_tL^2}\in(0,1/2)$ which implies
\begin{equation*}
    \norm{\srho \delta u}_{L^2_T L^2} \leq C \int_0^T \norm{\srho \delta u}_{L^2_t L^2}\abs{\ln{\parent{\norm{\srho \delta u}_{L^2_t L^2}}}}^{1/2} \abs{\ln{t}}^{1/2}t^{1/2}\parent{\norm{\srho D_t \bar{u}}+\norm{\nabla D_t \bar{u}}}dt.
\end{equation*}

We define $\al(t)=\abs{\ln{t}}^{1/2}t^{1/2}\parent{\norm{\srho D_t \bar{u}}+\norm{\nabla D_t \bar{u}}}$ which is integrable, $\al(t)\in L^1_T$. In addition, the function $\apl{r}{r\abs{\ln(r)}^{1/2}}$ is non decreasing in  $0^+$ and $\int_0^{\frac{1}{2}}r^{-1}\abs{\ln(r)}^{-1/2}dr=+\infty$. Then we can apply Osgood's Lemma (Lemma 3.4, \cite{BCD}) to conclude that $\norm{\srho \delta u}_{L^2_T L^2}=0$ on $[0,T]$. Moreover, by \eqref{preD} $\norm{\delta \rho}_{\dot{H}^{-1}}=0$ on $[0,T]$. More explicit, we have that on $[0,T]$,
\begin{equation*}
         \delta \rho \equiv 0\Rightarrow \rho=\bar{\rho}, \quad
          \srho \delta u \equiv 0 .
\end{equation*}

Testing the momentum equation in \eqref{UNI} against $\delta u$, we obtain $\int_0^T \norm{\nabla \delta u}^2 dt=0$. Therefore, by Lemma \ref{Lemma3.6}, 
 for every $R>0$,
$$
\norm{\delta u}_{L^2(B_R)}\leq C(R) (\norm{\srho \delta u}+\norm{\nabla \delta u})=0,
$$

which implies $\delta u=0 \Rightarrow u=\bar{u}$ on $[0,T]\times \R^2$.

The uniqueness on the whole time interval $[0,\infty)$ can be obtained by a bootstrap argument.

\section{Proof of Theorem \ref{th2}}\label{sect:viscosity contrast}
We next focus on incompressible viscous flows with non-constant density and non-constant viscosity on the unit torus $\mathbb{T}^2$.

\subsection{Existence:} In the first part of this work, we established the existence of global-in-time solutions in the whole space $\mathbb{R}^2$ by the combination of mollification techniques and compactness arguments. We now adapt this approach to the setting of the unit torus $\mathbb{T}^2$. As before, we smooth the initial data and derive a priori estimates for the corresponding smooth solutions. A key step is obtaining an estimate ensuring that the velocity field is spatially Lipschitz, in $L^1$-in time, which is essential for global existence. The uniformity of these bounds with respect to the mollification parameter allows us to pass to the limit via compactness. We now proceed to derive the a priori estimates in the periodic setting.

\medskip
 \textbf{Step 1: }$\srho u \in L^\infty_T L^2$ and $ \smu\D u \in L^2_T L^2$.

\medskip
    Taking the scalar product of the momentum equation in \eqref{VJ} with $u$, it is possible to obtain
    \begin{equation*}
        \produ{\rho u}{\partial_t u}+\produ{\rho \matt{u}}{u}=-\produ{\mu \D u}{\nabla u}.
    \end{equation*}
    Therefore
    \begin{equation*}
        \frac{1}{2}\frac{d}{dt}\norm{\srho u}^2=-\produ{\mu \D u}{\nabla u} = -\frac{1}{2}\norm{\smu \D u}^2.
    \end{equation*}
    After integrating in time we obtain the classical weak energy inequality \eqref{energy_ineq}.
    
\medskip
\textbf{Step 2: }  $t^{\frac{1-(\ga+s)}{2}}\sqrt{\rho}D_t u\in L^2_T L^2$ and 
        $t^{\frac{1-(\ga+s)}{2}}\nabla{u}\in L^\infty_T L^2$.

\medskip
   
We consider the linearized problem
    \begin{equation}\label{VJL}
\left\{
    \begin{aligned}
         \rho (\mat{v}) &= \nabla \c \parent{\mu \D v - \I P},   \\
         \dr&=0, \\
         \nabla \c v &=0.
    \end{aligned}
    \right.
\end{equation}
Analogously, $v$ satisfies \eqref{energy_ineq}, 
\begin{equation}
        \norm{\srho v}^2+\int_0^t\norm{\smu \D v}^2d\tau \leq \norm{\sr v_0}^2.
    \end{equation}

Testing the momentum equation of \eqref{VJL} with $\dv=\mat{v}$ and integrating by parts, it is possible to obtain
\begin{equation*}
    \norm{\srho \dv}^2=\int_{\T^2}\nabla \c (\mu \D v - \I P)\dv dx=-\int_{\T^2}(\mu \D v - \I P ) \nabla \dv dx.
\end{equation*}
    By the conmmutators, 
    $$
    \corch{D_t,\partial_j}f=-\partial_j u \c \nabla f,
    $$
    and the incompressibility condition we have
    \begin{equation*}
    \begin{split}
        \norm{\srho \dv}^2&= -\int_{\T^2}(\mu \D v - \I P ) (D_t \partial_j v_i + \partial_j u_k \partial_k v_i) dx
        \\
        & = -\int_{\T^2}\mu \D_{i,j} v D_t (\partial_j v_i)dx  -\int_{\T^2}(\mu \D v - \I P )\partial_j u_k \partial_k v_i dx.
    \end{split}
    \end{equation*}
    Since $\dm=0$, it follows that
    \begin{equation*}
        \int_{\T^2}\mu \D_{i,j} v D_t (\partial_j v_i)dx=\frac{1}{2}\frac{d}{dt}\norm{\smu \D v}^2.
    \end{equation*}
    Then, we obtain
    \begin{equation}\label{preP}
        \norm{\srho \dv}^2 + \frac{1}{2}\frac{d}{dt}\norm{\smu \D v}^2 = -\int_{\T^2}(\mu \D v - \I P )\partial_j u_k \partial_k v_i dx.
    \end{equation}
    To bound the right-hand side of the above equality we use the explicit expression of the pressure. Taking divergence in the momentum equation in \eqref{VJL} we obtain
    \begin{equation}\label{P}
          P=(-\Delta)^{-1}\nabla \c (\rho \dv) - \nabla \c \nabla (-\Delta)^{-1}(\mu \D v).
    \end{equation}

    Introducing \eqref{P} into \eqref{preP} we get
    \begin{equation}\label{4.7}
    \begin{split}
        \norm{\srho \dv}^2 + \frac{1}{2}\frac{d}{dt}\norm{\smu \D v}^2 =& - \int_{\T^2}\mu \D_{i,j} v\partial_j u_k \partial_k v_i + \int_{\T^2}(-\Delta)^{-1}\nabla \c (\rho \dv)\partial_j u_k \partial_k v_i dx
        \\
        &-\int_{\T^2}\nabla \c \nabla (-\Delta)^{-1}(\mu \D v)\partial_j u_k \partial_k v_i dx = I_1+I_2+I_3.
      \end{split}
    \end{equation}

We aim to obtain a $L^p$-estimate of $\nabla v$. By the equality
$$
\nabla \c (\mu \D v)=\bar{\mu} \Delta v + \nabla \c (\mu \D v-\bar{\mu}\D v),
$$
we can obtain 
\begin{equation}\label{expnu}
   \nabla v = \frac{1}{\bar{\mu}}\nabla \Delta^{-1}\P\nabla\c(\mu \D v)-\nabla \Delta^{-1} \P \nabla \Big[\big(\frac{\mu}{\bar{\mu}}-1\big)\D v\Big]. 
\end{equation}

In view of the smallness of $\mu$ in Theorem \ref{th2},
$$
\normp{\nabla v}{p}\leq c(\delta)\normp{\nabla \Delta^{-1}\P\nabla\c(\mu \D v)}{p}.
$$

Applying $\P$ in \eqref{VJL} we get 
\begin{equation}\label{P_rel}
    \P(\rho \dv)=\P \nabla \c (\mu \D v).
\end{equation} 

The above equality, jointly with the Gagliardo-Nirenberg inequality and the boundedness of singular integral operators (SIO) in $L^p$, $1<p<\infty$, leads to
\begin{equation}\label{grp}
    \normp{\nabla v}{p}\leq c(\delta) \norm{\smu \D v}^{2/p}\norm{\srho \dv}^{1-2/p}.
\end{equation}

In particular, for $p=4$,
\begin{equation}\label{grp4}
    \normp{\nabla v}{4}\leq c(\delta) \norm{\smu \D v}^{1/2}\norm{\srho \dv}^{1/2}.
\end{equation}

Going back to \eqref{4.7}, using Hölder inequality and \eqref{grp4} we obtain
\begin{equation}\label{I_1+I_3}
    \begin{split}
        I_1+I_3&\leq C \normp{\smu \D v}{4}\norm{\nabla u}\normp{\nabla v}{4} \leq C \norm{\nabla u}\norm{\smu \D v}\norm{\srho D_t v}\\
        &\leq \frac{1}{4}\norm{\srho D_t v}^2 + C \norm{\nabla u}^2\norm{\smu \D v}^2.
    \end{split}
\end{equation}

To bound $I_2$, we use the duality $BMO$-$\HH^1$, the div-curl lemma, Sobolev embeddings, and the boundedness in $L^p$ of SIO:
\begin{equation}\label{4.11}
    \begin{split}
        I_2&\leq C \norm{\partial_j u_k \partial_k v_i}_{\HH^1}\norm{(-\Delta)^{-1}\nabla \c (\rho \dv)}_{BMO}
        \leq C \norm{\nabla u}\norm{\nabla v}\norm{\srho \dv} \\
        &\leq \frac{1}{4} \norm{\srho \dv}^2+C \norm{\nabla u}^2\norm{\nabla v}^2.
    \end{split}
\end{equation}
Joining the bounds for $I_1$, $I_2$ and $I_3$, \eqref{4.7} becomes
\begin{equation*}
    \norm{\srho \dv}^2+\frac{d}{dt}\norm{\smu \D v}^2\leq C \norm{\nabla u}^2\norm{\smu \D v}^2.
\end{equation*}
Introducing the weight $t$ in the previous inequality gives
\begin{equation*}
    t\norm{\srho \dv}^2+\frac{d}{dt}\parent{t\norm{\smu \D v}^2}\leq C t\norm{\nabla u}^2\norm{\smu \D v}^2+\norm{\smu \D v}^2.
\end{equation*}

We now apply Gr\"onwall Lemma and \eqref{energy_ineq}, which yields
\begin{equation}\label{4.12}
    \int_0^t \tau \norm{\srho \dv}^2dt+ t \norm{\smu \D v}^2\leq C \exp{\Big(\int_0^t\norm{\nabla u}^2d\tau\Big)}\int_0^t\norm{\smu \D v}^2d\tau \leq C \norm{v_0}^2.
\end{equation}
On the other hand, if $v_0\in H^1(\T^2)$ we can apply the Gr\"onwall lemma straightforwardly in \eqref{4.11},
\begin{equation}\label{4.13} \int_0^t\norm{\srho \dv}^2d\tau+\norm{\smu \D v}^2\leq C \exp{\Big(\int_0^t\norm{\nabla u}^2d\tau}\Big)\norm{\sqrt{\mu_0}~\D v_0}^2\leq C \norm{v_0}_{H^1}^2.
\end{equation}

Interpolating between both inequalities, we obtain
\begin{equation}\label{lemma4.3}
\begin{split}
    \int_{0}^{t}\tau^{1-(\gamma+s)}\norm{\srho \du}^2d\tau+ t^{1-(\ga+s)}\norm{\smu \D u}^2 &\leq C \norm{u_0}^2_{H^{\ga+s}}.
\end{split}
\end{equation}

\medskip
\textbf{Step 3: } $t^{\frac{2-(\ga+s)}{2}}\sqrt{\rho}D_t u\in L^\infty_T L^2$ and $t^{\frac{2-(\ga+s)}{2}}\nabla \du \in L^2_TL^{2}$.

\medskip

 We apply $D_t$ to the momentum equation in \eqref{VJ} and test it against $\du$:
    \begin{equation*}
        \int_{\T^2}\du \c \rho D^2_t udx=\int_{\T^2}\du \c D_t\nabla\c(\mu \D u)dx + \int_{\T^2}\du \c D_t\nabla P.
    \end{equation*}
Using conmmutators again we have
\begin{equation}\label{4.17}
    \begin{split}
        \frac{1}{2}\frac{d}{dt}\norm{\srho \du}^2&=\int_{\T^2}D_t u_i \partial_j D_t (\mu \D_{i,j}u)dx - \int_{\T^2}D_t u_i \partial_j u_k \partial_k (\mu \D_{i,j}u)dx-\int_{\T^2}\du\c D_t \nabla P dx 
        \\
        &= I_4+I_5+I_6.
    \end{split}
\end{equation}
We start with $I_4$. Commutators and the fact that $\dm=0$ gives
\begin{equation*}
    \begin{split}
        I_4&=\int_{\T^2}D_t u_i \partial_j D_t (\mu \D_{i,j}u)dx = - \int_{\T^2}\partial_j D_t u_i D_t (\mu \D_{i,j}u)dx=-\int_{\T^2}\partial_j D_t u_i \mu (D_t\partial_j u_i + D_t \partial_i u_j)dx \\
        &=-\int_{\T^2}\partial_j D_t u_i \mu (\partial_jD_t u_i +\partial_i D_t  u_j)dx + \int_{\T^2}\partial_j D_t u_i \mu (\partial_ju_k\partial_ku_i+\partial_iu_k\partial_ku_j)dx
        \\
        &\leq -\frac{1}{2}\norm{\smu \D (\du)}^2+ C \norm{\mu \D (\du)}\normp{\nabla u}{4}^2.
    \end{split}
\end{equation*}

Now using \eqref{grp4} it follows that
\begin{equation*}
\begin{split}
        I_4&\leq -\frac{1}{2}\norm{\smu \D (\du)}^2 + C \norm{\mu \D (\du)} \norm{\srho \du} \norm{\smu \D u} 
        \\
        &\leq -\frac{9}{20}\norm{\smu \D (\du)}^2+ C \norm{\srho \du}^2 \norm{\smu \D u}^2.
\end{split}
\end{equation*}

Secondly, we bound $I_5$. Integration by parts and \eqref{grp4} give us
\begin{equation*}
    \begin{split}
        I_5& =  \int_{\T^2}\partial_kD_t u_i \partial_j u_k  \mu \D_{i,j}udx \leq C \norm{\smu \nabla \du}\normp{\nabla u}{4}\normp{\D u}{4}\\
        & \leq C \norm{\smu \nabla \du} \norm{\srho \du}\norm{\smu \D u}\leq \frac{1}{20\mu^M} \norm{\smu \nabla \du}^2 + C \norm{\srho \du}^2\norm{\smu \D u}^2.
    \end{split}
\end{equation*}

We are interested in obtaining a relationship between $\D \du$ and $\nabla \du$. 
We have the following equalities
\begin{equation*}
\begin{array}{c}
      \partial_k f_i = \partial_k \Delta^{-1} \partial_j \D_{i,j}f-\nabla \c \Delta^{-1} \partial_k\partial_i f,
      \qquad
     \nabla \c \du = \nabla u \c \nabla u.
\end{array}
\end{equation*}

Hence,
\begin{equation}\label{4.15}
    \begin{split}
        \norm{\nabla \du}^2&\leq \norm{\D (\du)}^2 + \norm{\nabla \c \du}^2\leq \norm{\D (\du)}^2+\norm{\nabla u \c \nabla u}^2 
        \\
        &\leq \norm{\D (\du)}^2+C\normp{\nabla u}{4}^4 
        \leq \norm{\D (\du)}^2 +C\norm{\smu \D u}^2\norm{\srho \du}^2.
    \end{split}
\end{equation}

Now, using \eqref{4.15} in $I_5$, 
\begin{equation*}
    \begin{split}
        I_5 \leq \frac{1}{20}\norm{\D (\du)}^2 +C\norm{\smu \D u}^2\norm{\srho \du}^2.
    \end{split}
\end{equation*}

Joining $I_4$ and $I_5$,

\begin{equation*}
    I_4+I_5 \leq -\frac{3}{10}\norm{\D (\du)}^2 + C \norm{\smu \D u}^2\norm{\srho \du}^2.
\end{equation*}

Lastly, we study $I_6$. From the analysis of \eqref{J_3} ($I_6 \sim J_3$) it is possible to obtain
\begin{equation*}
    I_6=-\int_{\T^2}\du D_t \c\nabla Pdx=\frac{d}{dt}\int_{\T^2}P\partial_i u_j \partial_j u_i dx - 3 \int_{\T^2} P \partial_j u_i \partial_i D_t u_jdx.
\end{equation*}

The pressure formula \eqref{P}, div-curl lemma and Sobolev embeddings gives us
\begin{equation*}
    \begin{split}
        \Big|\int_{\T^2} P \partial_j u_i \partial_i D_t u_jdx\Big|& \leq \Big|\int_{\T^2} \Delta^{-1}\nabla (\rho \du)\partial_j u_i \partial_i D_t u_jdx\Big| + \int_{\T^2}\abs{\nabla\c\nabla\Delta^{-1}(\mu \D u)}\abs{\partial_j u_i \partial_i D_t u_j}dx
        \\
        &\leq C \norm{\partial_j u_i \partial_i D_t u_j}_{\HH^1}\norm{\Delta^{-1}\nabla (\rho \du)}_{BMO}+C\normp{\nabla u}{4}\normp{\mu \D u}{4}\norm{\nabla (\du)}
        \\
        & \leq C \norm{\nabla u} \norm{\nabla \du} \norm{\srho \du}+ C \norm{\nabla \du} \norm{\srho \du}\norm{\smu \D u}
        \\
        & \leq \frac{1}{20}\norm{\nabla \du}^2 + C \norm{\srho \du}^2\norm{\smu \D u}^2.
    \end{split}
\end{equation*}

Introducing the previous bound in $I_6$ and using \eqref{4.15} we get
\begin{equation*}
    I_6 \leq \frac{d}{dt}\int_{\T^2}P\partial_i u_j \partial_j u_i dx +\frac{3}{20}\norm{\D (\du)}^2 + C \norm{\srho \du}^2\norm{\smu \D u}^2.
\end{equation*}
Therefore, \eqref{4.17} can be rewritten as
\begin{equation*}
    \frac{d}{dt}\Big(\frac{1}{2}\norm{\srho \du}^2-\int_{\T^2}P\partial_i u_j \partial_j u_i dx\Big) + \frac{1}{4} \norm{\smu \D (\du)}^2 \leq C \norm{\smu \D u}^2\norm{\srho \du}^2.
\end{equation*}
We denote by $\ph(t)= \frac{1}{2}\norm{\srho \du}^2-\displaystyle \int_{\T^2}P\partial_i u_j \partial_j u_i dx$. Introducing the time-wight $t^{2-(\ga+s)}$ we obtain
\begin{equation*}
    \begin{split}
        \frac{d}{dt}\parent{t^{2-(\ga+s)}\ph(t)} + \frac{1}{4}t^{2-(\ga+s)}\norm{\smu \D (\du)}^2\leq C t^{2-(\ga+s)} \norm{\smu \D u}^2\norm{\srho \du}^2 + t^{1-(\ga+s)}\ph(t).
    \end{split}
\end{equation*}
Using the bounds \eqref{4.11} and \eqref{I_1+I_3} with $v=u$ we have
\begin{equation*}
 \Big|\int_{\T^2}P\partial_i u_j \partial_j u_i dx\Big| \leq I_2+I_3 \leq \frac{1}{4}\norm{\srho \du}^2 + C \norm{\nabla u}^2\norm{\smu \D u}^2.
\end{equation*}
The above estimate allows us to derive upper and lower bounds for $\ph(t)$, as it was done in \eqref{ULBound},
\begin{equation}\label{boundphi}
    \frac{1}{4}\norm{\srho \du}^2 - C \norm{\nabla u}^2\norm{\smu \D u}^2 \leq \ph(t) \leq \frac{3}{4}\norm{\srho \du}^2 + C \norm{\nabla u}^2\norm{\smu \D u}^2.
\end{equation}
Therefore, using the lower bound of $\ph(t)$ we have
\begin{equation*}
\begin{split}
    \frac{d}{dt}\parent{t^{2-(\ga+s)}\ph(t)} + \frac{1}{4}t^{2-(\ga+s)}\norm{\smu \D (\du)}^2\leq& C t^{2-(\ga+s)} \norm{\smu \D u}^2\ph(t) \\
    &+C t^{2-(\ga+s)} \norm{\smu \D u}^4 \norm{\nabla u}^2 + t^{1-(\ga+s)}\ph(t).
\end{split}
\end{equation*}
 Grönwall Lemma yields
\begin{equation*}
    \begin{split}
        t^{2-(\ga+s)}&\ph(t)+ \frac{1}{4}\int_0^t \tau^{2-(\ga+s)}\norm{\smu \D (\du)}^2 d\tau \\
        &\leq C \exp{\Big(\int_0^t\norm{\smu \D u}^2d\tau\Big)} \Big(C \int_0^t \tau^{2-(\ga+s)} \norm{\smu \D u}^4 \norm{\nabla u}^2 d\tau + \int_0^t \tau^{1-(\ga+s)}\ph(\tau) d\tau\Big) .
    \end{split}
\end{equation*}

Using the lower bound of $\ph(t)$ it follows that
\begin{equation*}
\begin{split}
        \frac{1}{4}t^{2-(\ga+s)}&\norm{\srho \du}^2+\frac{1}{4}\int_0^t \tau^{2-(\ga+s)}\norm{\smu \D (\du)}^2 d\tau \\
        \leq&  C \exp{\Big(\int_0^t\norm{\smu \D u}^2 d\tau\Big)} \Big(C \int_0^t \tau^{2-(\ga+s)} \norm{\smu \D u}^4 \norm{\nabla u}^2 d\tau + \int_0^t t^{1-(\ga+s)}\ph(\tau) d\tau \Big) \\
        &+ t^{2-(\ga+s)}\norm{\smu \D u}^4.
\end{split}
\end{equation*}

Lastly, by \eqref{energy_ineq} and \eqref{lemma4.3}, together with the upper bound for $\ph(t)$ we can conclude that
\begin{equation}\label{lemma4.4}
   t^{2-(\ga+s)} \norm{\srho \du}^2+\int_0^t \tau^{2-(\ga+s)}\norm{\smu \D (\du)}^2 d\tau \leq C(T, \norm{u_0}_{H^{\ga+s}},\norm{\sr u_0}).
\end{equation}

\medskip
\textbf{Step 4: }
The solution near $t=0$ is controlled by the time-weighted estimates. However, when considering estimates \eqref{lemma4.3} and \eqref{lemma4.4} away from $t=0$, there is no need to introduce a time weight, since for fixed $t>0$, we have $\nabla u(t)\in\LpT{2}$ by \eqref{lemma4.3}. In this step, we focus on the decay in time of the solution. To this end, we study the solution in the time interval $(1,t)$ as $t\to\infty$.

\begin{lemma}\label{ThDecay}
     Let $(\rho, u)$ a solution to \eqref{VJ} for $t\geq1$ (recall that $u(1)\in \LpT{2}$) given by Theorem \ref{th2}. There exists a constant $C>0$ depending on $u(1)$, and $0<\beta_3<\beta_2<\beta_1=\frac{C\mu^m}{2\rho^M\norm{\rho}^2}$ such that  $\forall t\geq 1$ we have
     \begin{equation}\label{D1}
         \norm{\srho u}\leq C e^{-\beta_1 t},
     \end{equation}
     \begin{equation}\label{D2}
             e^{2\beta_2 t}\norm{\smu \D u}^2 + 
            \displaystyle C\int_1^t e^{2\beta_2\tau}\parent{\norm{\srho u}^2+\norm{\smu \D u}^2 +  \norm{\srho \du}^2}d\tau \leq C,
     \end{equation}
     \begin{equation}\label{D3}
         e^{2\beta_3 t}\norm{\srho \du}^2+C\int_1^t e^{2\beta_3 \tau} \norm{\smu \D (\du)}^2d\tau\leq C.
     \end{equation}
\end{lemma}

\begin{proof}
    We start with \eqref{D1}. As in the a priori estimates we test the momentum equation of \eqref{VJ} with $u$ and we get
\begin{equation}\label{4.24}
    \frac{1}{2}\frac{d}{dt}\norm{\srho u}^2 + \int_{\T^2}\mu\abs{ \D u}^2 dx = 0.
\end{equation}

By the lower bound of $\mu$ and Poincaré-type inequality \eqref{Poincaré},
\begin{equation*}
    \begin{split}
        \frac{d}{dt}\norm{\srho u}^2 + \frac{C\mu^m}{2\rho^M\norm{\rho}^2} \int_{\T^2}\rho^M\abs{u}^2dx \leq \frac{d}{dt}\norm{\srho u}^2 + \mu^m \int_{\T^2}\abs{ \D u}^2 dx \leq 0.
    \end{split}
\end{equation*}
By the definition of upper bound of $\rho$ and denoting by $\beta_1=\frac{C\mu^m}{2\rho^M\norm{\rho}^2}$,
\begin{equation*}
     \frac{d}{dt}\norm{\srho u}^2+2\beta_1\norm{\srho u}^2 \leq 0 ~\Longrightarrow ~\frac{d}{dt}\norm{\srho u}^2\leq -2\beta_1\norm{\srho u}^2.
\end{equation*}

Integrating in $[1,\infty)$, we obtain the result
\begin{equation*}
    \norm{\srho u}^2 \leq Ce^{-2\beta_1 t} \|\sqrt{\rho(1)}u(1)\|^2 \leq C(\norm{\sr u_0}) e^{-2\beta_1 t}, 
\end{equation*}

Next, to establish \eqref{D2}, we test the equation with $\du$  in the same manner as in \eqref{lemma4.3}. It follows that
\begin{equation*}
    \frac{d}{dt}\norm{\smu \D u}^2+\norm{\srho \du}^2 \leq C \norm{\nabla u}^4.
\end{equation*}

We add \eqref{4.24} to the above inequality to get
\begin{equation*}
    \frac{d}{dt}\parent{\norm{\srho u}^2+\norm{\smu \D u}^2}+\norm{\srho \du}^2 + \norm{\smu \D u}^2\leq C \norm{\nabla u}^4.
\end{equation*}

Adding $\norm{\srho u}^2$ in both sides of the above inequality and taking $\beta_2\in\R$ (conditions on $\beta_2$ imposed later) and using \eqref{D1} we obtain 
\begin{equation}\label{eqbeta2}
\begin{split}
    \frac{d}{dt}&\parent{e^{2\beta_2t}\norm{\srho u}^2+e^{2\beta_2t}\norm{\smu \D u}^2}  \\
    & + (1-2\beta_2)e^{2\beta_2t}\parent{\norm{\smu \D u}^2 + \norm{\srho u}^2 + \norm{\srho \du}^2}  \leq C e^{2\beta_2t} \norm{\nabla u}^4 + C e^{-2(\beta_1-\beta_2)t} .
\end{split}
\end{equation}

Then, $\beta_2$ must satisfy
$$
\begin{array}{lr}
    1-2\beta_2>0~\Longrightarrow~\beta_2<1/2, &  \quad \beta_1-\beta_2>0 ~\Longrightarrow~\beta_2<\beta_1
\end{array}
$$

Hence we choose $\beta_2=\min\llav{1/4,\beta_1/2}$.

Adding $e^{2\beta_2t} \norm{\nabla u}^2\norm{\srho u}^2$ in both sides in \eqref{eqbeta2} and rewriting the right-hand side, we find
\begin{equation*}
    \begin{split}
    \frac{d}{dt}
    &\parent{e^{2\beta_2t}\norm{\srho u}^2+e^{2\beta_2t}\norm{\smu \D u}^2}+Ce^{2\beta_2t}\parent{ \norm{\smu \D u}^2 + \norm{\srho \du}^2 + \norm{\srho u}^2} \\ 
    & + e^{2\beta_2t} \norm{\nabla u}^2\norm{\srho u}^2
    \leq C \norm{\nabla u}^2 \parent{e^{2\beta_2t}\norm{\nabla u}^2 + e^{2\beta_2t}\norm{\srho u}^2} +C e^{-2(\beta_1-\beta_2)t}.
\end{split}
\end{equation*}

Applying Gr\"onwall's Lemma in $(1,t)$, 
\begin{equation*}
\begin{split}
    e^{2\beta_2t}\norm{\srho u}^2&+e^{2\beta_2t}\norm{\smu \D u}^2 +C\int_1^te^{2\beta_2\tau}\parent{\norm{\srho u}^2 + \norm{\smu \D u}^2 + \norm{\srho \du}^2}d\tau
    \\ 
    & \leq C \exp\Big(\int_1^t\norm{\nabla u}^2d\tau\Big)\Big( \|\sqrt{\rho(1)}u(1)\|^2+\|\sqrt{\mu(1)} \D u(1)\|^2+C\int_1^t e^{-2(\beta_1-\beta_2) \tau}d\tau\Big).
\end{split}
\end{equation*}

Thus we can conclude that
\begin{equation*}
    e^{2\beta_2t}\norm{\srho u}^2+e^{2\beta_2t}\norm{\smu \D u}^2 + C \int_1^te^{2\beta_2\tau}\parent{\norm{\srho u}^2 + \norm{\smu \D u}^2 + \norm{\srho \du}^2}d\tau \leq C.
\end{equation*}

Lastly, we prove \eqref{D3}. As in the proof of \eqref{lemma4.4} we apply $D_t$ to the momentum equation in \eqref{VJ} and test it with $\du$,
\begin{equation*}
    \frac{d}{dt}\ph(t)+\frac{1}{4}\norm{\smu \D (\du)}^2 \leq C \norm{\nabla u}^2 \norm{\srho \du}^2,
\end{equation*}
where $\ph(t)=\displaystyle\frac{1}{2}\norm{\srho\du}^2-\int_{\T^2}P\partial_ju_i\partial_iu_jdx$ and we recall that $\ph(t)$ satisfies the upper and lower bounds in \eqref{boundphi}. 
Multiplying the equation by $e^{2\beta_3t}$, where $\beta_3=\beta_2/2<1/4$,  and then using  \eqref{D2}, we get
\begin{equation*}
\begin{split}
    \frac{d}{dt}\parent{e^{2\beta_3t}\ph(t)}+\frac{e^{2\beta_3t}}{4}\norm{\smu \D (\du)}^2  &\leq C e^{2\beta_3t} \norm{\nabla u}^2 \norm{\srho \du}^2 + 2\beta_3 e^{2\beta_3t}\ph(t)\\
    &\leq C e^{-2(\beta_2-\beta_3)t} \norm{\srho \du}^2 + C e^{2\beta_3t}\ph(t).
\end{split}
\end{equation*}

Integrating in $[1,\infty)$ we have
\begin{equation}\label{preD3}
\begin{split}
    e^{2\beta_3t}\ph(t) + \frac{1}{4}\int_1^t e^{2\beta_3\tau} \norm{\smu \D (\du)}^2 d \tau \leq C &+ C\int_1^t e^{-2(\beta_2-\beta_3)\tau} \norm{\srho \du}^2 d\tau\\&+C\int_1^t e^{2\beta_3\tau}\ph(\tau)d\tau.
\end{split}
\end{equation}
The estimate \eqref{D2} and the fact that $\beta_2-\beta_3>0$ implies that the first integral in the right hand side is bounded uniformly in time. Using the upper bound of $\ph$ in \eqref{boundphi} followed by \eqref{D2}, the second integral is bounded as well,
\begin{equation*}
\begin{split}
    \int_1^t e^{2\beta_3\tau}\ph(\tau)d\tau&\leq \frac{3}{4} \int_1^t e^{2\beta_3\tau}\norm{\srho \du}^2d\tau + C \int_1^t e^{2\beta_3\tau}\norm{\nabla u}^4 d\tau \\
    & \leq \frac{3}{4} \int_1^t e^{2\beta_2\tau}\norm{\srho \du}^2d\tau + C \int_1^t e^{-2(2\beta_2-\beta_3)\tau} d\tau \leq C.
\end{split}
\end{equation*}
Therefore, introducing the above estimates in \eqref{preD3} and using the lower bound of $\ph(t)$ in \eqref{boundphi}, we get
\begin{equation*}
    e^{2\beta_3 t}\norm{\srho \du(t)}^2+C\int_1^t e^{2\beta_3 \tau} \norm{\smu \D (\du(\tau))}^2d\tau \leq C + C e^{2\beta_3 t }\norm{\nabla u}^4 \leq C,
\end{equation*}
where we have used \eqref{D2} in the last inequality.
\end{proof}

\medskip
\textbf{Step 5: } 
Once we have obtained the decay estimates, we focus on proving the integrability in time of $\normp{\nabla u}{\infty}$. We start by rewriting $\nabla u$ using \eqref{expnu}.
\begin{equation}\label{decompu}
    \nabla u =\frac{1}{\bar{\mu}}\nabla \Delta^{-1}\P (\rho \du)-\nabla \Delta^{-1} \P \nabla \Big(\big(\frac{\mu}{\bar{\mu}}-1\big)\D u\Big) = I_7+I_8.
\end{equation}

We use Sobolev embeddings, the equality \eqref{P_rel} and the boundedness in $L^p$ of SIO, to bound $I_7$,
\begin{equation*}
\begin{split}
    \abs{I_7}&\leq C \norm{\nabla \Delta^{-1}\P (\rho \du)}_{W^{1, \frac{2}{1-\ga}}} \leq C \normp{\nabla \Delta^{-1}\P (\rho \du)}{\frac{2}{1-\ga}}+C\normp{\nabla \nabla \Delta^{-1}\P (\rho \du)}{\frac{2}{1-\ga}}
    \\
    &\leq C \normp{\nabla \Delta^{-1}\P \nabla \c (\mu \D u)}{\frac{2}{1-\ga}}+C\normp{\nabla \nabla \Delta^{-1}\P (\rho \du)}{\frac{2}{1-\ga}} \leq C \normp{\mu \D u}{\frac{2}{1-\ga}}+C\normp{\rho \du}{\frac{2}{1-\ga}}.
\end{split}
\end{equation*}

By \eqref{grp} and Lemma \ref{LemDanchin2} for $p=\frac{2}{1-\ga}$
\begin{equation*}
    \begin{split}
        \abs{I_7}&\leq  C \norm{\nabla u}^{1-\ga}\norm{\srho \du}^\ga + C \norm{\srho \du}^{1-\ga}\norm{\nabla \du}^\ga \log^{\ga/2}\Big(e+\frac{\norm{\rho-M}^2}{M^2}+\frac{\norm{\nabla \du}^2}{\norm{\srho \du}^2}\Big)\\
        &=I_{7,1}+I_{7,2}. 
    \end{split}
\end{equation*}
We divide
\begin{equation*}
    \int_0^T \abs{I_{7,1}}dt = \int_0^1 \abs{I_{7,1}}dt + \int_1^T \abs{I_{7,2}}dt,
\end{equation*}
in order to use different properties of the solution in each time regime.

On $(0,1)$, we apply \eqref{lemma4.3} and \eqref{lemma4.4} to obtain
\begin{equation}\label{aux1}
\begin{aligned}
    \int_0^1 \abs{I_{7,1}}dt &= C\int_0^1 \norm{\nabla u}^{1-\ga}\norm{\srho \du}^\ga dt \\
    &\leq C\int_0^1 t^{-\frac{(1-\ga)(1-(\ga+s))}{2}}t^{-\ga(1-\frac{\ga+s}{2})}dt= C\int_0^1 t^{-\frac{1-s}{2}}dt\leq C.
\end{aligned}
\end{equation}

On $(1,T)$, using the decay estimates \eqref{D2} and \eqref{D3}, we get

\begin{equation*}
    \int_1^T \abs{I_{7,1}}dt = C\int_1^T \norm{\nabla u}^{1-\ga}\norm{\srho \du}^\ga dt \leq C\int_0^Te^{-t(\beta_2(1-\ga)+\beta_3\ga)}dt \leq C.
\end{equation*}

We proceed analogously, splitting
\begin{equation}\label{I7,2}
    \int_0^T \abs{I_{7,2}} dt = \int_0^1 \abs{I_{7,2}} dt + \int_1^T \abs{I_{7,2}} dt.
\end{equation}

To handle the logarithmic term, we notice that for any $\ep>0$,
\begin{equation}\label{log_bound}
\log(e + x^2) \lesssim x^{2\ep} \quad \text{for all } x \gg 1,
\end{equation}
which allows us to absorb the logarithmic factor into a power of the norm.
Hence, using \eqref{lemma4.4} and Young's inequality for $p=\frac{2}{\ga(1+\ep)}$, 
\begin{equation}\label{aux2}
    \int_0^1 \abs{I_{7,2}} dt \leq C\int_0^1 \norm{\srho \du}^{1-\ga\parent{1+\ep}}\norm{\nabla \du}^{\ga\parent{1+\ep}}  dt \leq C\int_0^1 t^{-\frac{2-(\ga+s)}{2-\ga(1+\ep)}}dt \leq C,
\end{equation}
which is finite if we take $\ep < s$.
Decay estimate \eqref{D3} and Young's inequality with $p=\frac{2}{\ga(1+\ep)}$ give us
\begin{equation*}
    \int_1^T \abs{I_{7,2}} dt \leq C\int_1^T \norm{\srho \du}^{1-\ga\parent{1+\ep}}\norm{\nabla \du}^{\ga\parent{1+\ep}}  dt \leq C\int_1^T e^{-\beta_3\frac{2-\ga(1+\ep)}{2}t}dt \leq C,
\end{equation*}
which is finite if we take $\ep<2/\ga-1$.
Recall that for $x \ll 1$, the logarithm is bounded and we can also obtain a similar bound. Therefore, we can conclude that there exists $C>0$ independent of time such that
\begin{equation}\label{I7bound}
    \int_0^T \abs{I_7}dt\leq C.
\end{equation}

The bound of $I_8$ will be obtained in a completely different manner from $I_7$. We are interested in taking advantage of the initial regularity of the patch, which will be conserved in time. Using the definition of the Leray projector $\P$ it is possible to obtain
\begin{equation*}
\begin{split}
    I_8&=\nabla \Delta^{-1} \P \nabla \c \Big(\big(\frac{\mu}{\bar{\mu}}-1\big)\D u\Big)=\nabla \Delta^{-1} \nabla \c \parent{\I-\nabla \Delta^{-1} \nabla\c} \Big(\big(\frac{\mu}{\bar{\mu}}-1\big)\D u\Big)
\end{split}
\end{equation*}

The operator $\nabla \Delta^{-1} \P \nabla$ is a Singular Integral Operator on the torus $\T^2$, consisting of second and fourth order Riesz Transform. 
As $\big(\frac{\mu}{\bar{\mu}}-1\big)\D u$ is a $2\pi$-periodic function in $L^2$ we can write 
\begin{equation*}
\big(\frac{\mu}{\bar{\mu}}-1
\big)\D u(x)= \sum_{l\in\Z^2} c_l e^{i l\c x},
\end{equation*}
where $c_l= \FF\big(\big(\frac{\mu}{\bar{\mu}}-1\big)\D u\big)(l)$ is the coefficient of the Fourier series. Hence, we have
\begin{equation*}
    I_8=\sum_{l\in\Z^2} \Big(\frac{l_il_m}{\abs{l}^2}-\frac{l_il_jl_kl_m}{\abs{l}^4}\Big) c_l e^{i l\c x}=\int_{\T^2}K^*(x-y)\big(\frac{\mu(y)}{\bar{\mu}}-1\big)\D u(y)dy,
\end{equation*}
where $K^*(x)$ is the kernel obtained by periodizing the kernel associated to the Riesz Transforms in $\R^2$  (see \cite{Calderon54}). As expected, $K^*(x)$ has the same properties as the analogue in the whole space. 
Then, we conclude by following the strategy developed in \cite{Gancedo23} to estimate the term $I_{8}$. 
The approach relies on analyzing the behavior of the function separately inside and outside $D(t)$, 
reducing the singular behavior to the boundary. The singular term is then bounded thanks to the $C^{1+\gamma}$ regularity of the boundary, together with the fact the kernels of the singular operators are even.
On each side, one can take advantage of the piecewise-Hölder regularity of the viscosity and $\mathbb{D}u$. In turn, this higher regularity must be propagated, and the control of $\|\nabla u\|_{L^1_TL^\infty}$  is again needed. The circular argument is overcome by exploiting the smallness of the viscosity jump, together with the uniform-in-time estimate achieved for $I_7$, where the decay estimate become essential. The use of the Gagliardo-Nirenberg inequality in \cite{Gancedo23} must be replaced here by its refined logarithmic version (Lemma \ref{LemDanchin2}), as in \eqref{I7,2}, allowing us to close the argument and obtain the desired estimates, 
\begin{equation*}
    \begin{aligned}
        \int_0^T \|\nabla u\|_{L^\infty}dt\leq C,\hspace{0.5cm}
        \int_0^T \|\nabla u\|_{\dot{C}^\gamma_D}dt\leq C.
    \end{aligned}
\end{equation*}

\subsection{Uniqueness}
In this case, uniqueness is established in Lagrangian coordinates. This approach allows for the difficulties arising from the low regularity of the density and viscosity.
We denote $(\rho_0,v)$ the solution to \eqref{VJ} in Lagrangian coordinates,
\begin{equation*}
    \begin{array}{cc}
        \rho_0(y)=\rho(t,X(t,y)), & v(t,y)=u(t,X(t,y)), 
    \end{array}
\end{equation*}
and correspondingly the pressure $Q$ and viscosity $\mu_0$,
\begin{equation*}
    \begin{array}{cc}
        Q(t,y)=P(t, X(t,y)), & \mu_0(y)=\mu(t,X(t,y)), 
    \end{array}
\end{equation*}
where $X$ denotes the flow defined by \eqref{traj}.
%\begin{equation*}X(t,y)=y +\int_0^t v (\tau, y) d\tau.\end{equation*}
Note that
\begin{equation*}
    \nabla X(t,y) = \I+\int_0^t \nabla v (\tau,y)d\tau.
\end{equation*}
We define $A=(\nabla X(\c,t))^{-1}$ and $A^*$ its transpose. Then, Lagrangian variables provide 
\begin{equation*}
\nabla_u = A^* \nabla,\quad\nabla_u \c=\div(A \,)= A^* : \nabla,\quad 
    \D_{u}=A^*\nabla  + (\nabla \,)^* A. %\quad\Delta_u = \div(A^*A\nabla \,).
\end{equation*}
where the divergence-free condition is used in the second equality. The system \eqref{VJ} becomes \begin{equation*}
\left \{
\begin{aligned}
    \rho_0 \partial_t v &= \nabla_{u}(\mu_0 \D_{u}v-Q\I),\\
    \nabla_{u}\c v&=0,
    \end{aligned}
\right.
\end{equation*}
being equivalent if 
\begin{equation*}
    \int_0^T \norm{\nabla v}_{L^\infty} \leq c <1.
\end{equation*}
In that case, we can write
\begin{equation}\label{A_def}
    A(t)=\sum_{j=0}^\infty (-1)^j\Big(\int_0^t \nabla v(\tau, \c) d \tau\Big)^j.
\end{equation}
Let $(\rho^1,u^1)$, $(\rho^2,u^2)$ be two solutions to \eqref{VJ} for the same initial data. In Lagrangian coordinates, their difference will be denoted as
$$
\begin{array}{lcr}
   \delta v=v^2-v^1,  & \delta Q = Q^2-Q^1, & \delta A= A^2-A^1,
\end{array}
$$
and the following equation is satisfied 
\begin{equation*}
    \left\{
    \begin{aligned}
         \rho_0  \partial_t \delta v &= \nabla_{u^1} \c (\mu_0 \D_{u^1}\delta v-\delta Q \I) + \nabla_{u^2} \c (\mu_0 \D_{u^2} v^2- Q^2 \I)-\nabla_{u^1} \c (\mu_0 \D_{u^1} v^2- Q^2 \I), \\
         \nabla_{u^1} \c \delta v &= (\nabla_{u^1}-\nabla_{u^2})\c v^2,
         \\
         \delta v |_{t=0}&=0.
    \end{aligned}
    \right.
\end{equation*}
We will now prove that for $T>0$ small enough
\begin{equation*}
    \int_0^T \norm{\nabla \delta v}^2 dt=0.
\end{equation*}
We split $\delta v = z + w$ where $w$ is the solution to the equation
\begin{equation}\label{w}
    \nabla_{u^1} \c  w = (\nabla_{u^1}-\nabla_{u^2})\c v^2 = \nabla \c (\delta A v^2).
\end{equation}
Then the equation for $z$ becomes
\begin{equation}\label{SIS_z}
 \left\{ \begin{aligned}
 \rho_0 \partial_t z - \nabla_{u^1} \c \left( \mu_0 \D_{u^1} z \right) &= \nabla_{u^1} \delta Q + \nabla_{u^2} \c \left( \mu_0 \D_{u^2} v^2 - Q^2 \I \right) - \nabla_{u^1} \c ( \mu_0 \D_{u^1} v^2 - Q^2 \I ) \\
 &\quad- \rho_0 \partial_t w + \nabla_{u^1} \c \left( \mu_0 \D_{u^1} w \right), \\ \nabla_{u^1} \c \delta z &= 0, \\
 \delta z |_{t=0} &= 0. 
 \end{aligned} \right.
\end{equation}
Note that 
\begin{equation}\label{deltaA}
    \norm{\delta A}\leq Ct^{1/2}\norm{\nabla \delta v}_{L^2_T L^2},
\end{equation}
and 
\begin{equation}\label{deltaA_t}
    \|\partial_t\delta A\|\leq C\|\nabla\delta v\|.
\end{equation}
To control $w$ we state the following lemma. 
\begin{lemma}\label{Lw}
    The solution $w$ to \eqref{w} satisfies
    \begin{equation*}
        \norm{w}_{L^\infty_T L^2}+\norm{\nabla w}_{L^2_T L^2} + \norm{\partial_t w}_{L^p_T L^r} \leq C(T) \norm{\nabla\delta v}_{L^2_T L^2},
    \end{equation*}
    where $p=\frac{2}{2-\ga}$, $q>2$, and $r=\frac{2q}{2+q}$. 
\end{lemma}
\begin{proof}
    Since $q>2$,  using \eqref{grp} and the estimates \eqref{lemma4.3}-\eqref{lemma4.4}, we have 
\begin{equation*}
    \begin{aligned}
        \|\nabla v^1\|_{L^q}\leq Ct^{-1+\frac{\gamma+s}{2}+\frac{1}{q}}.
    \end{aligned}
\end{equation*}
Therefore,
    \begin{equation*}
        \norm{\nabla v^1}_{L^1_T L^\infty}+\norm{\nabla v^1}_{L^p_T L^q}\leq C(T),
    \end{equation*}
    where $C(T)$ can be made sufficiently small by choosing $T$ small enough.
   Then we can apply Lemma \ref{Control_w} with $R=\delta A v^2$ and $g=\div(\delta A v^2)=\delta A^* : \nabla v^2 $.
 We first have
 \begin{equation*}
    \norm{w}\leq C \norm{\delta A v^2}. 
 \end{equation*}
Inequalities \eqref{deltaA} and \eqref{grp} provide \begin{equation*}
    \begin{split}
    \norm{\delta A v^2} &\leq C \norm{\nabla \delta v}_{L^2_T L^2}t^{1/2}\norm{v^2}_{L^\infty} \leq C \norm{\nabla \delta v}_{L^2_T L^2}t^{1/2}\norm{v^2}_{W^{1,\frac{2}{1-(\ga+s)}}}
        \\
        & \leq C \norm{\nabla \delta v}_{L^2_T L^2}t^{1/2}\norm{\nabla u^2}_{L^\frac{2}{1-(\ga+s)}}\\
        &\leq C \norm{\nabla \delta v}_{L^2_T L^2}\norm{t^{\frac{1-(\gamma+s)}2} \nabla u^2}^{1-(\ga+s)}\norm{t^{\frac{2-(\gamma+s)}2} \sr D_t u^2}^{\ga+s}. 
    \end{split}
    \end{equation*}
Bounds \eqref{lemma4.3} and \eqref{lemma4.4} yield
\begin{equation}\label{aux3}
\norm{\delta A v^2}_{L^\infty_T L^2}\leq C(T) \norm{\nabla \delta v}_{L^2_T L^2},
\end{equation}
    and thus
    $$\norm{w}_{L^\infty_T L^2}\leq C(T) \norm{\nabla \delta v}_{L^2_T L^2}.$$
    Next, using \eqref{deltaA} again yields
\begin{equation*}
        \norm{\nabla w}_{L^2_T L^2} \leq C \norm{\nabla \delta A^*: \nabla v^2}_{L^2_T L^2} \leq C \norm{\nabla \delta v}_{L^2_T L^2}\norm{t^{1/2}\nabla v^2}_{L^2_T L^\infty}.
    \end{equation*}
Step 5 can be repeated to get 
\begin{equation}\label{aux4}
\int_0^T t \norm{\nabla {u^2}}_{L^\infty}^2 dt\leq C.    
\end{equation}
In fact, one only needs to check the terms for small time, such as \eqref{aux1} or \eqref{aux2}, which follow directly by noticing  that if $\alpha<1$ then $2\alpha-1<1$.
Therefore, 
\begin{equation*}
        \norm{\nabla w}_{L^2_T L^2} \leq C(T) \norm{\nabla \delta v}_{L^2_T L^2}.
    \end{equation*}    
    Lastly, the estimate for the time derivatives in Lemma \ref{Control_w} gives  
    \begin{equation*}
    \begin{aligned}
  \norm{\partial_t w}_{L^p_T L^r} &\leq \norm{\delta A v^2}_{L^\infty_T L^2} + C \big(\norm{ \partial_t\delta A v^2}_{L^p_T L^r}+\norm{\delta A \partial_tv^2}_{L^p_T L^r}\big).
    \end{aligned}
    \end{equation*}
Estimate \eqref{aux3} and Hölder's inequality give
    \begin{equation*}
    \begin{aligned}
  \norm{\partial_t w}_{L^p_T L^r} &\leq C(T) \norm{\nabla \delta v}_{L^2_T L^2} + C \big(\norm{ \partial_t\delta A}_{L^2_T L^2} \norm{v^2}_{L^\frac{2}{1-\gamma}_T L^q}+\norm{\norm{\delta A} \norm{\partial_tv^2}_{L^q}}_{L^p_T}\big),
    \end{aligned}
    \end{equation*}
where it was used that $\frac{1}{r}=\frac{1}{2}+\frac{1}{q}$ and $p=\frac{2}{2-\gamma}$.
Then, \eqref{deltaA}-\eqref{deltaA_t} give
    \begin{equation*}
    \begin{aligned} 
        \norm{\partial_t w}_{L^p_T L^r}  & \leq C(T) \norm{\nabla \delta v}_{L^2_T L^2} +  C \norm{\nabla \delta v}_{L^2_T L^2} \Big(\norm{v^2}_{L^{\frac{2}{1-\ga}}_T L^q}+\norm{t^{1/2} \partial_t v^2}_{L^p_T L^q}\Big). 
           \end{aligned}
    \end{equation*}     
Recalling \eqref{Poincaré} and writing $\partial_t v^2=D_t v^2-v^2\cdot\nabla v^2$, we have
    \begin{equation*}
    \begin{aligned} 
        \norm{v^2}_{L^{\frac{2}{1-\ga}}_T L^q}&\leq C\norm{\nabla v^2}_{L^{\frac{2}{1-\ga}}_T L^2}\leq C(T),\\
        \norm{t^{1/2} \partial_t v^2}_{L^p_T L^q}&\leq \norm{t^{1/2} D_t v^2}_{L^p_T L^q}+\norm{t^{1/2} v^2\cdot \nabla v^2}_{L^p_T L^q}\leq C\norm{t^{1/2} \nabla D_t v^2}_{L^p_T L^2}+\norm{t^{1/2} v^2\cdot \nabla v^2}_{L^p_T L^q}\\
        &\leq C(T),
           \end{aligned}
    \end{equation*}     
thus
    \begin{equation*}
    \begin{aligned} 
        \norm{\partial_t w}_{L^p_T L^r}&\leq C(T) \norm{\nabla \delta v}_{L^2_T L^2}.
    \end{aligned}
    \end{equation*}
\end{proof}
Next, to obtain a control over $z$ we test the momentum equation in \eqref{SIS_z} against $z$ and we obtain
\begin{equation}\label{z}
    \frac{d}{dt}\norm{\sr z}^2+\norm{\smuo \D_{u^1}z}^2 \leq\tilde{I}_1 + \tilde{I}_2 + \tilde{I}_3 + \tilde{I}_4,
\end{equation}
where
\[
\begin{tabular}{@{}l@{\hspace{1cm}}l@{}}
$\tilde{I}_1 = \int_{\T^2} z \c \left(\nabla_{u^2} \c (\mu_0 \D_{u^2}v^2) - \nabla_{u^1} \c (\mu_0 \D_{u^1}v^2)\right) dy,$
&
$\tilde{I}_2 = \int_{\T^2} z \c \left(\nabla_{u^1} Q^2 - \nabla_{u^2} Q^2 \right) dy,$
\\[1.2em]
$\tilde{I}_3 = -\int_{\T^2} \rho_0 z\cdot\, \partial_t w \, dy,$
&
$\tilde{I}_4 = -\int_{\T^2} z\cdot\nabla_{u^1} \c (\mu_0 \D_{u^1} w) \, dy.$
\end{tabular}
\]

We proceed bounding each of these terms. Hence,
\begin{equation*}
    \begin{split}
  \tilde{I}_1 &= \int_{\T^2} \nabla z\corch{A^2\mu_0\parent{(A^2)^*\nabla v^2 + (\nabla v^2)^* A^2}-A^1\mu_0\parent{(A^1)^*\nabla v^2 + (\nabla v^2)^* A^1}} dy
        \\
        & = \int_{\T^2}\mu_0 \nabla z \corch{(\delta A (A^2)^*+A^1\delta A^*)\nabla v^2+\delta A (\nabla v^2)^* A^2 + A^1 (\nabla v^2)^*\delta A}dy
        \\
        & \leq C \norm{\delta A} \norm{\nabla z}\norm{\nabla v^2}_{L^\infty}.
    \end{split}
\end{equation*}
Integrating in time, it is possible to obtain
\begin{equation*}
    \begin{split}
        \int_0^T |\tilde{I}_1| dt \leq  C \norm{\nabla \delta v}_{L^2_T L^2} \norm{t^{1/2} \nabla v^2}_{L^2_T L^\infty}\norm{\nabla z}_{L^2_T L^2}
        \leq C(T) \norm{\nabla \delta v}_{L^2_T L^2}^2 + \frac{1}{8}\norm{\nabla z}_{L^2_T L^2}^2.
    \end{split}
\end{equation*}
Using \eqref{P} we can rewrite $Q^2$ as follows
\begin{equation*}
    Q^2=P^2(X^2)=(-\Delta)^{-1}\nabla \c (\rho^2 D_t u^2)(X^2)-\nabla \c \nabla (-\Delta)^{-1} (\mu^2\D u^2)(X^2).
\end{equation*}
Introducing this expression of $Q^2$ in $\tilde{I}_2$ and integration by parts in the second integral yield 
\begin{equation*} 
    \begin{split}
\tilde{I}_2=& \int_{\T^2}\delta A z (\nabla \Delta^{-1}\nabla\c)(\rho^2D_t u^2)(X^2)\nabla X^2dy
+\int_{\T^2}\delta A^* : \nabla z (\nabla \c \nabla \Delta^{-1})(\mu^2 \D{u^2})(X^2)dy
        \\
         \leq&  C\norm{\delta A}\norm{z}_{L^{2/\ga}}\norm{\sqrt{\rho^2}D_t u^2}_{L^\frac{2}{1-\ga}}\|\nabla X^2\|_{L^\infty}+
         C \norm{\delta A}\norm{\nabla z} \norm{(\nabla \c \nabla \Delta^{-1})(\mu^2 \nabla {u^2})}_{L^\infty}\\
        \leq& C\norm{\delta A}(\|\rho_0z\|+\norm{\nabla z})\norm{\sqrt{\rho^2}D_t u^2}^{1-\ga}\norm{\nabla D_t u^2}^\ga \log^{\ga/2}\Big(e+\frac{\norm{\nabla D_t u^2}^2}{\norm{\sqrt{\rho^2}D_t u^2}}\Big)(1+\|\nabla v^2\|_{L^1_TL^\infty})\\
        &+C \norm{\delta A}\norm{\nabla z} \norm{(\nabla \c \nabla \Delta^{-1})(\mu^2 \nabla {u^2})}_{L^\infty},
    \end{split}
\end{equation*}
where we have used Hölder inequality, together with Lemmas \ref{PoincaréFull} and \ref{LemDanchin2}. Integrating in time and using \eqref{deltaA} we obtain \begin{equation*}
    \begin{split}
        \int_0^T |&\tilde{I}_2| dt\leq C\norm{\nabla \delta v}_{L^2_TL^2}\norm{\nabla z}_{L^2_T L^2}\norm{t^{1/2} (\nabla \c \nabla \Delta^{-1})(\mu^2 \nabla {u^2})}_{L^2_T L^\infty} 
        \\
        + &C\norm{\nabla \delta v}_{L^2_TL^2}(\|\rho_0z\|+\norm{\nabla z}_{L^2_T L^2}) \Big\|t^{1/2}\norm{\sqrt{\rho^2}D_t u^2}^{1-\ga}\norm{\nabla D_t u^2}^\ga \log^{\ga/2}\Big(e\!+\!\frac{\norm{\nabla D_t u^2}^2}{\norm{\sqrt{\rho^2}D_t u^2}}\Big)\Big \|_{L^2_T}.
    \end{split}
\end{equation*}
From \eqref{aux4}, we  have that $t^{1/2}(\nabla \c \nabla \Delta^{-1})(\mu^2 \nabla {u^2})\in L^2_T L^\infty$. Then, standard bounds for the logarithm, as in \eqref{I7,2}, allow us to conclude 
\begin{equation*}
    \begin{split}
        \int_0^T |\tilde{I}_2|dt \leq C(T)\norm{\nabla \delta v}^2_{L^2_T L^2}+\frac{1}{8}\norm{\rho_0 z}^2_{L^2_T L^2}+\frac{1}{8}\norm{\nabla z}^2_{L^2_T L^2}.
    \end{split}
\end{equation*}
We proceed now with $\tilde{I}_3$. By Lemma \ref{Lw} with $p=\frac{2}{2-\ga}$, $q>\frac{2}{\gamma}$ and $r=\frac{2q}{2+q}$, and $p', r'$ the dual exponents,
\begin{equation*}
    \begin{split}
        \int_0^T |\tilde{I}_3|dt &\leq C \norm{\partial_t w}_{L^p_T L^r}\norm{\sr z}_{L^{p'}_T L^{r'}} \leq C(T) \norm{\nabla \delta v}_{L^2_T L^2}\norm{\sr z}_{L^{p'}_T L^{r'}}.
    \end{split}
\end{equation*}
Using Lemma \ref{LemDanchin2} it is possible to get
\begin{equation*}
    \begin{split}
        \int_0^T \norm{\sr z}_{L^{r'}}^{p'}dt &\leq \int_0^T \norm{\sr z}^{\frac{2(q-2)}{\ga q}}\norm{\nabla z}^{\frac{4}{\ga q}}\log^{\frac{2}{\ga q}}\Big(e+\frac{\norm{\nabla z}^2}{\norm{\sr z}^2}\Big)dt.
    \end{split}
\end{equation*}
We estimate the term above for the case when the log term is bounded,  otherwise, taking into account \eqref{log_bound}, it suffices to slightly modify the exponents.
Hölder's inequality  gives
\begin{equation*}
    \begin{split}
        \int_0^T \norm{\sr z}_{L^{r'}}^{p'}dt & \leq C\left (\int_0^T \norm{\sr z}^{\frac{2(q-2)}{\ga q-2}}dt\right)^{\frac{\ga q-2}{\ga q}}\left (\int_0^T \norm{\nabla z}^2 dt\right)^{\frac{2}{\ga q}} \\
        & \leq C(T) \norm{\sr z}_{L^\infty_T L^2}^{\frac{2(q-2)}{\ga q}}\norm{\nabla z}_{L^2_T L^2}^{\frac{4}{\ga q}}.
    \end{split} 
\end{equation*}

Introducing the bound above in $\tilde{I_3}$, recalling that $p'=\frac{2}{\gamma}$, yields
\begin{equation*}
    \begin{split}
        \int_0^T |\tilde{I}_3|dt &\leq C(T)\norm{\nabla \delta v}_{L^2_T L^2}\norm{\sr z}_{L^\infty_T L^2}^{\frac{q-2}{ q}}\norm{\nabla z}_{L^2_T L^2}^{\frac{2}{ q}} \\ &\leq C(T)\norm{\nabla \delta v}_{L^2_T L^2}^2 + \frac{1}{8} \norm{\sr z}_{L^\infty_T L^2}^2 + \frac{1}{8} \norm{\nabla z}_{L^2_T L^2}^2.
    \end{split}
\end{equation*}

The last term, $\tilde{I}_4$ is bounded using Hölder inequality and Lemma \ref{Lw} to obtain
\begin{equation*}
    \begin{split}
        \int_0^T |\tilde{I}_4|dt\leq C(T) \norm{\nabla \delta v}^2_{L^2_T L^2}+\frac{1}{8}\norm{\nabla z}^2_{L^2_T L^2}.
    \end{split}
\end{equation*}

Hence, integrating in time in \eqref{z} and joining all the bounds give
\begin{equation}\label{z.2}
    \norm{\sr z}^2_{L^\infty_T L^2}+\norm{\nabla z}^2_{L^2_T L^2} \leq C(T) \norm{\nabla \delta v}_{L^2_T L^2}^2. 
\end{equation}
Recall that $\delta v = w+z$, Lemma \ref{Lw} and \eqref{z.2} yields
\begin{equation*}
     \norm{\nabla \delta v}_{L^2_T L^2} \leq  \norm{\nabla z}_{L^2_T L^2} +  \norm{\nabla w}_{L^2_T L^2} \leq C(T)  \norm{\nabla \delta v}_{L^2_T L^2},
\end{equation*}
which implies that for $T \ll 1$, $ \norm{\nabla \delta v}_{L^2_T L^2}=0$. The energy estimates \eqref{z.2} yield $\norm{\sr z}_{L^\infty_T L^2} = 0=\norm{\nabla z}_{L^2_T L^2}$. Applying Lemma \ref{PoincaréFull} then leads to $\norm{z}_{L^2_T L^2} = 0$. From Lemma \ref{Lw}, $\norm{w}_{L^\infty_T L^2}=0$. Therefore, we can conclude that $v^1=v^2$ in $[0,T]\times \T^2$. We can now return to Eulerian coordinates, $u^1=u^2$, and the extension to any arbitrary $T>0$ follows from standard connectivity arguments. 

\section{Acknowledgements}

FG was partially supported by the Institute of Advanced Study and the Harish-Chandra Fund.
EGJ was partially supported by the RYC2021-032877 research grant (Spain). PLV was supported by an FPI grant from the Spanish Government PRE2022-105212. 
FG and EGJ were partially supported by the RED2022-134784-T funded by  MCIN/AEI/10.13039/50110001103. 
FG, EGJ and PLV were partially supported by the AEI project PID2022-140494NA-I00 (Spain). FG and PLV were partially supported by the Fundaci\'on de Investigaci\'on de la Universidad de Sevilla through the
grant FIUS23/0207 and acknowledge support
from IMAG, funded by MICINN through the Maria de Maeztu Excellence Grant CEX2020-001105-M/AEI/10.13039/501100011033.  

%--------------------------------------------------------------
%--------------------------------------------------------------

% References
\bibliographystyle{acm}
\bibliography{references}
\end{document}